\documentclass[12pt,a4paper]{amsart}

\usepackage{amsmath,amsthm,amsfonts,amssymb,array,amscd}
\usepackage{amsmath,geometry,amssymb,amsfonts,amsthm,graphicx,enumerate,latexsym,tabularx,amscd}
\usepackage[all,cmtip]{xy}
\usepackage{refcheck} 
\norefnames
\nocitenames
\usepackage{geometry}
\usepackage{amsthm}
 \newtheorem{thm}{Theorem}[section]
 \newtheorem{prop}[thm]{Proposition}
 \newtheorem{lem}[thm]{Lemma}
 \newtheorem{cor}[thm]{Corollary}
\theoremstyle{definition}
 \newtheorem{exm}[thm]{Example}
 \newtheorem{dfn}[thm]{Definition}
\theoremstyle{remark}
 \newtheorem{rem}[thm]{Remark}
 \numberwithin{equation}{section}
\theoremstyle{definition}
\theoremstyle{remark}
 \numberwithin{equation}{section}

\renewcommand{\le}{\leqslant}
\renewcommand{\ge}{\geqslant}\renewcommand{\geq}{\geqslant}

\usepackage{tikz}
\usepackage[pdftex]{hyperref}
\usetikzlibrary{matrix,arrows}

\newcommand{\bbC}{\mathbb{C}}

\newcommand{\bbQ}{\mathbb{Q}}
\newcommand{\bbR}{\mathbb{R}}

\newcommand{\bbZ}{\mathbb{Z}}   

\newcommand{\cN}{\mathcal{N}}

\renewcommand{\and}{\quad \mbox{and} \quad}  
\renewcommand{\le}{\leqslant}
\renewcommand{\ge}{\geqslant}\renewcommand{\geq}{\geqslant}

\title[Local constants for Heisenberg representations]{Local constants for Heisenberg representations}


\author[Biswas]{\bfseries Sazzad Ali Biswas}

\address{
School of Mathematics and Statistics\\ 
University of Hyderabad   \\ 
Hyderabad, 500046\\
India}
\email{sazzad.biswas@uohyd.ac.in, sazzad.jumath@gmail.com}

\thanks{The  author is partially supported by IMU-Berlin Einstein Foundation, Berlin, Germany and CSIR, Delhi, India }

\begin{document}

\vspace{18mm}
\setcounter{page}{1}
\thispagestyle{empty}

\begin{abstract}
We can attach a local constant to every finite dimensional continuous complex representation of a 
local Galois group of a non-archimedean local field $F/\bbQ_p$ by Deligne and Langlands.
Tate \cite{JT1} gives an explicit formula for computing local constants for linear characters of $F^\times$, 
but there is no explicit formula of local constant for any arbitrary representation of a local Galois group.
In this article we study Heisenberg representations of the absolute Galois group $G_F$ of $F$
and give invariant formulas of local constants for Heisenberg
representations of dimension prime to $p$.

\end{abstract}

\maketitle

\section{\textbf{Introduction}}

Let $F$ be a non-archimedean local field 
(i.e., finite extension of the $p$-adic field $\mathbb{Q}_p$, for some prime $p$). Let $\overline{F}$ be an algebraic closure
of $F$, and $G_F:=\rm{Gal}(\overline{F}/F)$ be the absolute Galois group of $F$. Let
\begin{center}
 $\rho:G_F\to \mathrm{Aut}_{\mathbb{C}}(V)$ 
\end{center}
be a finite dimensional continuous complex representation of the Galois group $G_F$. For this
$\rho$, we can associate a constant 
$W(\rho)$ with absolute value $1$ by Langlands (cf. \cite{RL}) and Deligne (cf. \cite{D1}).
This constant is called the \textbf{local constant} (also known as local epsilon factor) of the representation $\rho$. 
Langlands also proves that 
these local constants are weakly extendible functions (cf. \cite{JT1}, p. 105, Theorem 1).

The existence of this local constant is proved by Tate for one-dimensional representation in  
\cite{JT3} and 
the general proof of the existence of the local constants is proved by Langlands (see \cite{RL}). 
In 1972 Deligne also gave a proof  using global methods in \cite{D1}.
But in Deligne's terminology this local constant $W(\rho)$ is 
 $\epsilon_{D}(\rho,\psi_F,\mathrm{dx},1/2)$, where $\mathrm{dx}$ is the Haar
 measure on $F^{+}$ (locally compact abelian group) which is self-dual with respect to the {\bf canonical}
 (i.e., coming through trace map from $\psi_{\bbQ_p}(x):=e^{2\pi ix}$ for all $x\in\bbQ_p$, see \cite{JT1}, p. 92)
 additive character $\psi_F$ of $F$. 
Tate in his article \cite{JT2} denotes this Langlands convention of local constants as $\epsilon_{L}(\rho,\psi)$. 
According to Tate (cf. \cite{JT2}, p. 17),
 the Langlands factor $\epsilon_{L}(\rho,\psi)$ is 
 $\epsilon_{L}(\rho,\psi)=\epsilon_{D}(\rho\omega_{\frac{1}{2}},\psi,\mathrm{dx_{\psi}})$, where $\omega$ denotes the normalized 
 absolute value of $F$, i.e., $\omega_{\frac{1}{2}}(x)=|x|_{F}^{\frac{1}{2}}=q_{F}^{-\frac{1}{2}\nu_{F}(x)}$ 
 which we may consider as a character of $F^\times$, and where
 $\mathrm{dx_{\psi}}$ is the self-dual Haar measure corresponding to the additive character $\psi$ and 
 $q_F$ is the cardinality of the residue field of $F$. According to Tate (cf. \cite{JT1}, p. 105) 
 the relation among three conventions of the local constants is:
 \begin{equation}
  W(\rho)=\epsilon_{L}(\rho,\psi_F)=\epsilon_{D}(\rho\omega_{\frac{1}{2}},\psi_F,\mathrm{dx_{\psi_F}}).
 \end{equation}

In Section 2, we discuss all the necessary notations and known results for this article. In Section 3 we study the arithmetic
description of Heisenberg representations and their 
determinants (cf. Proposition \ref{Proposition arithmetic form of determinant}) of the absolute Galois group $G_F$ of a 
local field $F/\bbQ_p$. In particular, the Heisenberg representations of dimension prime to $p$ are important for this article.
In Subsection 3.2, we study the various properties (e.g., Artin conductors, Swan conductors, dimension)
of Heisenberg representations of dimension prime to $p$.

In Section 4, firstly, we give an invariant formula of local constant for a Heisenberg representation 
$\rho$ of the absolute Galois group 
$G_F$ of a non-archimedean local field $F/\bbQ_p$ (cf. Theorem \ref{Theorem invariant odd}).
In Theorem \ref{invariant formula for minimal conductor representation}, we give an invariant formula
of local constant of a minimal conductor Heisenberg representation $\rho$ of dimension prime to $p$. And when 
$\rho$ is not minimal conductor but dimension is prime to $p$, we have Theorems \ref{Theorem invariant for non minimal representation},
\ref{Theorem using Deligne-Henniart}.

In Section 5, we also discuss Tate's root-of-unity criterion, and by applying this Tate's criterion we give some information about 
the dimension and Artin conductor of a Heisenberg representation (cf. Proposition \ref{Proposition 4.12}).

\section{\textbf{Notations and Preliminaries}}

\subsection{Abelian Local Constants}

We have explicit formula of abelian local constants due to Tate (cf. \cite{JT1}, pp. 93-94). 
Let $F$ be a non-archimedean local field.
Let $O_F$ be the 
ring of integers of the local field $F$
and $P_F=\pi_F O_F$ be a prime ideal in $O_F$, 
where $\pi_F$ is a uniformizer, i.e., an element in $P_F$ whose valuation is one, i.e.,
 $v_F(\pi_F)=1$.
The order of the residue field 
of $F$ is $q_F$. Let $U_F=O_F-P_F$ be the group of units in $O_F$.
Let $P_{F}^{i}=\{x\in F:v_F(x)\geq i\}$ and for $i\geq 0$ define $U_{F}^{i}=1+P_{F}^{i}$
(with proviso $U_{F}^{0}=U_F=O_{F}^{\times}$).

Let $\chi$ be a character of $F^\times$ with conductor $a(\chi)$, i.e., the smallest integer such that $\chi$ is trivial on 
$U_{F}^{a(\chi)}$. Let $\psi$ be an additive character of $F$ with conductor $n(\psi)$, i.e., $\psi$ is trivial on $P_{F}^{-n(\psi)}$,
nontrivial on $P_{F}^{-n(\psi)-1}$. Then the local constant of $\chi$ is (cf. \cite{JT1}, p. 94):
\begin{equation}\label{eqn 2.9}
 W_F(\chi,\psi)=\chi(c)q_{F}^{-\frac{a(\chi)}{2}}\sum_{x\in \frac{U_F}{U_{F}^{a(\chi)}}}\chi^{-1}(x)\psi(\frac{x}{c}),
\end{equation}
where $c=\pi_{F}^{a(\chi)+n(\psi)}$.

\begin{dfn}[\textbf{Different and Discriminant}] 
 Let $K/F$ be a finite separable extension of non-archimedean local field $F$. We define the \textbf{inverse different (or codifferent)}
 $\mathcal{D}_{K/F}^{-1}$ of $K$ over $F$ to be $\pi_{K}^{-d_{K/F}}O_K$, where $d_{K/F}$ is the largest integer (this is the 
 exponent of the different $\mathcal{D}_{K/F}$) such that 
 \begin{center}
  $\mathrm{Tr}_{K/F}(\pi_{K}^{-d_{K/F}}O_K)\subseteq O_F$,
 \end{center}
 where $\rm{Tr}_{K/F}$ is the trace map from $K$ to $F$.
Then the \textbf{different} is defined by:
\begin{center}
 $\mathcal{D}_{K/F}=\pi_{K}^{d_{K/F}}O_K$
\end{center}
and the \textbf{discriminant} $D_{K/F}$ is 
\begin{center}
 $D_{K/F}=N_{K/F}(\pi_{K}^{d_{K/F}})O_F$.
\end{center}
 If $K/F$ is 
tamely ramified, then 
\begin{equation}\label{eqn 2.2}
 \nu_K(\mathcal{D}_{K/F})=d_{K/F}=e_{K/F} - 1.
\end{equation}
\end{dfn}

 \subsection{Extendible functions}
 
Let $G$ be any finite group. We denote $R(G)$ the set of all pairs $(H,\rho)$, where $H$ is a subgroup of $G$ and $\rho$ is a
virtual representation of $H$
. The group $G$ acts on $R(G)$ by means of
\begin{center}
$(H,\rho)^g=(H^g,\rho^g)$,     $g\in G$,\\
$\rho^g(x)=\rho(gxg^{-1})$,   $x\in H^g:=g^{-1}Hg$
\end{center}
Furthermore we denote by $\widehat{H}$ the set of all one dimensional representations of $H$ and 
by $R_1(G)$ the subset of $R(G)$ of pairs $(H,\chi)$ with 
$\chi\in \widehat{H}$. Here character $\chi$ of $H$ we mean always a \textbf{linear} 
character, i.e., $\chi:H\to \mathbb{C}^\times$. 

Now define a function $\mathcal{F}:R_1(G) \rightarrow \mathcal{A}$, where $\mathcal{A}$ is a multiplicative abelian  group with
\begin{equation}
\mathcal{F}(H,1_H)=1\label{eqn 2.1}
\end{equation}
 and 
\begin{equation}
\mathcal{F}(H^g,\chi^g)=\mathcal{F}(H,\chi)\label{eqn 2.2}
\end{equation}
for all $(H,\chi)$, where $1_H$ denotes the trivial representation of $H$.\\
Here a function $\mathcal{F}$ on $R_1(G)$ means a function which satisfies the equation (\ref{eqn 2.1}) 
and (\ref{eqn 2.2}).

A function $\mathcal{F}$ is said to be extendible if $\mathcal{F}$ can be extended to 
an $\mathcal{A}$-valued 
function
on $R(G)$ satisfying: 
\begin{equation}\label{eqn 2.3}
 \mathcal{F}(H,\rho_1+\rho_2)=\mathcal{F}(H,\rho_1)\mathcal{F}(H,\rho_2)
\end{equation}
for all $(H,\rho_i)\in R(G),i=1,2$, and if $(H,\rho)\in R(G)$ with $\mathrm{dim}\,\rho=0$, and $\Delta$ is a subgroup of
$G$ containing 
$H$, then
\begin{equation}\label{eqn 2.4}
 \mathcal{F}(\Delta,\mathrm{Ind}_{H}^{\Delta}\rho)=\mathcal{F}(H,\rho),
\end{equation}
where $\mathrm{Ind}_{H}^{\Delta}\rho$ is the virtual representation of $\Delta$ induced from $\rho$. In general, 
let $\rho$ be a 
representation of $H$ with $\mathrm{dim}\,\rho\neq0$.
We can define a zero dimensional representation of $H$ by $\rho$ and which is:
  $\rho_0:=\rho-\mathrm{dim}\,\rho\cdot 1_H$. So $\mathrm{dim}\,\rho_0$ is zero, then now we use the equation (\ref{eqn 2.4}) for
$\rho_0$ and we have,
\begin{equation}\label{eqn 2.5}
 \mathcal{F}(\Delta,\mathrm{Ind}_{H}^{\Delta}\rho_0)=\mathcal{F}(H,\rho_0).
 \end{equation}
 Now replace $\rho_0$ by $\rho-\mathrm{dim}\rho\cdot 1_H$ in the above equation (\ref{eqn 2.5}) and we have
 \begin{align*}
   \mathcal{F}(\Delta,\mathrm{Ind}_{H}^{\Delta}(\rho-\mathrm{dim}\rho \cdot 1_H))
   &=\mathcal{F}(H,\rho-\mathrm{dim}\rho\cdot1_H)\\\implies
   \frac{\mathcal{F}(\Delta,\mathrm{Ind}_{H}^{\Delta}\rho)}
   {\mathcal{F}(\Delta,\mathrm{Ind}_{H}^{\Delta}1_H)^{\mathrm{dim}\rho}}
   &=\frac{\mathcal{F}(H,\rho)}
   {\mathcal{F}(H,1_H)^{\mathrm{dim}\rho}}.
 \end{align*}
Therefore,
\begin{align}
 \mathcal{F}(\Delta,\mathrm{Ind}_{H}^{\Delta}\rho)\nonumber
 &=\left\{\frac{\mathcal{F}(\Delta,\mathrm{Ind}_{H}^{\Delta}1_H)}{\mathcal{F}(H,1_H)}\right\}^{\mathrm{dim}\rho}\cdot\mathcal{F}(H,\rho)\\
 &=\lambda_{H}^{\Delta}(\mathcal{F})^{\mathrm{dim}\rho}\mathcal{F}(H,\rho), \label{eqn 2.6}
\end{align}
where
\begin{equation}\label{eqn 2.7}
 \lambda_{H}^{\Delta}(\mathcal{F}):=\frac{\mathcal{F}(\Delta,\mathrm{Ind}_{H}^{\Delta}1_H)}{\mathcal{F}(H,1_H)}.
\end{equation}
But by the definition of $\mathcal{F}$, we have 
$\mathcal{F}(H,1_H)=1$, so we can write 
\begin{equation}\label{eqn 2.8}
 \lambda_{H}^{\Delta}(\mathcal{F})={\mathcal{F}(\Delta,\mathrm{Ind}_{H}^{\Delta}1_H}).
\end{equation}
This $\lambda_{H}^{\Delta}(\mathcal{F})$
is called \textbf{Langlands $\lambda$-function} (or simply $\lambda$-function) which is independent of $\rho$.
A extendible function $\mathcal{F}$ is called \textbf{strongly} extendible if it satisfies
equation (\ref{eqn 2.3}) and fulfills equation (\ref{eqn 2.4}) for all $(H,\rho)\in R(G)$, and if the equation (\ref{eqn 2.4})
is fulfilled
only when $\mathrm{dim}\,\rho=0$, then
$\mathcal{F}$ is called \textbf{weakly} extendible function. The extendible functions are \textbf{unique}, if they exist 
(cf. \cite{JT1}, p. 103).

\begin{exm}
 Langlands proves the local constants are weakly extendible functions (cf. \cite{JT1}, p. 105, Theorem 1). The Artin root numbers
(also known as global constants) 
are strongly 
extendible functions (for more examples and details about extendible function, see \cite{JT1} and \cite{HK}).
\end{exm}


Now we take a tower of local Galois extensions $K/L/F$, and denote $G=\rm{Gal}(K/F)$, $H=\rm{Gal}(K/L)$. Then the $\lambda$-function
for the extension $L/F$ is:
$$\lambda_{\rm{Gal}(K/L)}^{\rm{Gal}(K/F)}(W):=\lambda_{L/F}(\psi)=W(\rm{Ind}_{L/F}(1_L),\psi),$$
where $1_L$ is the trivial character of $L^\times$ which corresponds to the trivial character of $H$ by class field theory,
and $\psi$ is a nontrivial additive character of $F$. And when we take $\psi=\psi_F$ as the canonical additive character, we simply write
$\lambda_{L/F}$ instead of $\lambda_{L/F}(\psi_F)$.

Since the Heisenberg representations of a finite local Galois are monomial (i.e., induced from linear character
of a finite-index subgroup), we need to know the explicit
formula for lambda functions for finite Galois extensions. For this article we need the following computations of 
lambda functions.
\begin{thm}[\cite{SAB1}, Theorem 3.5]\label{General Theorem for odd case} 
Let $F$ be a non-archimedean local field and $\mathrm{Gal}(E/F)$ be a local Galois group of odd order. If 
$L\supset K\supset F $ be any finite extension inside $E$, then $\lambda_{L/K}=1$. 
\end{thm}


\begin{thm}[\cite{SAB1}, Theorem 5.9]\label{Theorem 3.21}
 Let $K$ be a tamely ramified quadratic extension of $F/\bbQ_p$ with $q_F=p^s$. Let $\psi_F$ be the canonical additive character of $F$.
 Let $c\in F^\times$ with $-1=\nu_F(c)+d_{F/\bbQ_p}$, and $c'=\frac{c}{\rm{Tr}_{F/F_0}(pc)}$, where $F_0/\bbQ_p$ is the maximal unramified
 extension in $F/\bbQ_p$. Let $\psi_{-1}$ be an additive character of $F$ with conductor $-1$, of the form $\psi_{-1}=c'\cdot\psi_F$.
 Then 
 \begin{equation*}
  \lambda_{K/F}(\psi_F)=\Delta_{K/F}(c')\cdot\lambda_{K/F}(\psi_{-1}),
 \end{equation*}
where 
 \begin{equation*}
 \lambda_{K/F}(\psi_{-1})=\begin{cases}
                                               (-1)^{s-1} & \text{if $p\equiv 1\pmod{4}$}\\
                                                  (-1)^{s-1}i^{s} & \text{if $p\equiv 3\pmod{4}$}.
                                            \end{cases}
\end{equation*}
If we take $c=\pi_{F}^{-1-d_{F/\bbQ_p}}$, where $\pi_F$ is a norm for $K/F$, then 
\begin{equation}
 \Delta_{K/F}(c')=\begin{cases}
                   1 & \text{if $\overline{\rm{Tr}_{F/F_0}(pc)}\in k_{F_0}^{\times}=k_{F}^{\times}$ is a square},\\
                   -1 & \text{if $\overline{\rm{Tr}_{F/F_0}(pc)}\in k_{F_0}^{\times}=k_{F}^{\times}$ is not a square}.
                  \end{cases}
\end{equation}
Here "overline" stands for modulo $P_{F_0}$.
\end{thm}

\begin{thm}[\cite{SAB1}, Corollary 4.11(1)]\label{Theorem 2.5}
 Let $G=\rm{Gal}(E/F)$ be a finite local Galois group of a non-archimedean local field $F/\bbQ_p$ with $p\ne 2$. 
 Let $S\cong G/H$ be a nontrivial Sylow 2-subgroup of $G$, where $H$ is a uniquely determined Hall subgroup of odd order. Suppose that 
 we have a tower $E/K/F$
 of fields such that $S\cong \rm{Gal}(K/F)$, $H=\rm{Gal}(E/K)$ and $G=\rm{Gal}(E/F)$.
  If $S\subset G$ is cyclic, then 
  \begin{enumerate}
 \item
 \begin{equation*}
 \lambda_{1}^{G}=\lambda_{K/F}^{\pm 1}=\begin{cases}
                  \lambda_{K/F}=W(\alpha) & \text{if $[E:K]\equiv 1\pmod{4}$}\\
                  \lambda_{K/F}^{-1}=W(\alpha)^{-1} & \text{if $[E:K]\equiv -1\pmod{4}$},
                 \end{cases}
 \end{equation*}
 (here $\alpha=\Delta_{K/F}$ corresponds to the unique quadratic subextension in $K/F$) 
if $[K:F]=2$, hence $\alpha=\Delta_{K/F}$.
 \item 
 \begin{equation*}
\lambda_{1}^{G}=\beta(-1)W(\alpha)^{\pm 1}=\beta(-1)\times\begin{cases}
                  W(\alpha) & \text{if $[E:K]\equiv 1\pmod{4}$}\\
                  W(\alpha)^{-1} & \text{if $[E:K]\equiv -1\pmod{4}$}
                 \end{cases}
 \end{equation*}
if $K/F$ is cyclic of order $4$ with generating character $\beta$ such that 
 $\beta^2=\alpha=\Delta_{K/F}$.
 \item 
 \begin{equation*}
  \lambda_{1}^{G}=\lambda_{K/F}^{\pm 1}=\begin{cases}
                  \lambda_{K/F}=W(\alpha) & \text{if $[E:K]\equiv 1\pmod{4}$}\\
                  \lambda_{K/F}^{-1}=W(\alpha)^{-1} & \text{if $[E:K]\equiv -1\pmod{4}$}
                 \end{cases}
 \end{equation*}
 if $K/F$ is cyclic of order $2^n\ge 8$.
\end{enumerate}
  And if the $4$th roots of unity are in the $F$, we have the same formulas as above but with $1$ instead of $\pm 1$.
Moreover, when $p\ne 2$, a precise formula for $W(\alpha)$ will be obtained in Theorem \ref{Theorem 3.21}.

\end{thm}

\subsection{Classical Gauss sums}

Let $k_q$ be a finite field of order $q$. Let $\chi, \psi$ be a multiplicative and an additive character respectively of $k_q$. 
Then the Gauss sum $G(\chi,\psi)$ is 
defined
by 
\begin{equation}
 G(\chi,\psi)=\sum_{x\in k_{q}^{\times}}\chi(x)\psi(x).
\end{equation}
For this article we need the following theorem.
In general, we cannot give explicit formula of  $G(\chi,\psi)$ for arbitrary character $\chi$. But if $q=p^r$($r\ge 2$), where 
$p$ is an odd prime, then by R. Odoni  
(cf. \cite{BRK}, p. 33, Theorem 1.6.2) we can show that $G(\chi,\psi)/\sqrt{q}$ is a certain root of unity. If $q$ is an odd prime
and order of $\chi$ is $\ge 3$, then $G(\chi,\psi)/\sqrt{q}$ is {\bf not} a root of unity.
\begin{thm}[Chowla, \cite{BRK}, p. 31, Theorem 1.6.1]\label{Theorem Chowla}
 Let $q$ be an odd prime, and let $\chi$ be a character of $k_{q}^{\times}$ of order $>2$. Let 
 $\psi(x)=e^{\frac{2\pi i x}{q}}$ for $x\in k_q$. Then the Gauss sum $G(\chi,\psi)$
 does not equal to $\sqrt{q}$ times a root of unity,
\end{thm}



\subsection{\textbf{Heisenberg representation}}

Let $\rho$ be an irreducible representation of a (pro-)finite group $G$. Then $\rho$ is called a \textbf{Heisenberg 
representation} if it represents commutators by 
scalar matrices. Therefore higher commutators are represented by $1$.
We can see that the linear characters of $G$ are Heisenberg representations as the degenerate special case.
To classify Heisenberg representations we need to mention two invariants of an irreducible representation 
$\rho\in\rm{Irr}(G)$:
\begin{enumerate}
 \item Let $Z_\rho$ be the \textbf{scalar} group of $\rho$, i.e., $Z_\rho\subseteq G$ and $\rho(z)=\text{scalar matrix}$
for every $z\in Z_\rho$. If $V/\bbC$ is a representation space of $\rho$ we get $Z_\rho$ as the kernel of the composite map 
\begin{equation}\label{eqn 2.6.1}
 G\xrightarrow{\rho}GL_{\bbC}(V)\xrightarrow{\pi} PGL_{\bbC}(V)=GL_{\bbC}(V)/\bbC^\times E,
\end{equation}
where $E$ is the unit matrix and denote $\overline{\rho}:=\pi\circ\rho$.
Therefore $Z_\rho$ is a normal subgroup of $G$.
\item Let $\chi_\rho$ be the character of $Z_\rho$ which is given as $\rho(g)=\chi_\rho(g)\cdot E$ for all $g\in Z_\rho$. 
Apparently $\chi_\rho$ is a $G$-invariant character of $Z_\rho$ which we call the central 
character of $\rho$.
\end{enumerate}
Let $A$ be a profinite abelian group. Then we know that (cf. \cite{Z5}, p. 124, Theorem 1 and Theorem 2)
the set of isomorphism classes $\rm{PI}(A)$ of projective irreducible representations (for 
projective representation, see \cite{CR}, \S  51) of $A$ is in bijective correspondence with the 
set of continuous alternating characters $\rm{Alt}(A)$. If $\rho\in\rm{PI}(A)$ corresponds to $X\in\rm{Alt}(A)$ then 
\begin{center}
 $\rm{Ker}(\rho)=\rm{Rad}(X)$ \hspace{.4cm} and \hspace{.2cm}$[A:\rm{Rad}(X)]=\rm{dim}(\rho)^2$,
\end{center}
where $\rm{Rad}(X):=\{a\in A|\, X(a,b)=1,\,\text{for all}\, b\in A\}$, the {\bf radical of $X$}.

Let $A:=G/[G,G]$, so $A$ is abelian. 
We also know from the  composite map (\ref{eqn 2.6.1})
$\overline{\rho}$ is a projective irreducible representation of $G$ and $Z_\rho$ is the kernel of $\overline{\rho}$.
Therefore \textbf{modulo commutator group $[G,G]$}, we can consider that 
$\overline{\rho}$ is in $\rm{PI}(A)$ which corresponds an alternating 
character $X$ of $A$ with kernel of $\overline{\rho}$ is $Z_\rho/[G,G]=\rm{Rad}(X)$.
We also know that 
$$[A:\rm{Rad}(X)]=[G/[G,G]:Z_\rho/[G,G]]=[G:Z_\rho].$$
Then we observe that 
$$\rm{dim}(\overline{\rho})=\rm{dim}(\rho)=\sqrt{[G:Z_\rho]}.$$

Let $H$ be a subgroup of $A$, then we define the orthogonal complement of $H$ in $A$ with respect to $X$
$$H^\perp:=\{a\in A:\quad X(a, H)\equiv1\}.$$
An {\bf isotropic} subgroup $H\subset A$ is a subgroup such that $H\subseteq H^\perp$ (cf. \cite{EWZ}, p. 270, Lemma 1(v)).
And when isotropic subgroup $H$ is maximal,
we call $H$ is a \textbf{maximal isotropic} for $X$. Thus when $H$ is maximal isotropic we have 
$H=H^\perp$.

We also can show that the Heisenberg representations $\rho$ are fully characterized by the corresponding pair 
$(Z_{\rho},\chi_{\rho})$.

\begin{prop}[\textbf{\cite{Z3}, Proposition 4.2}]\label{Proposition 3.1}
The map $\rho\mapsto(Z_\rho,\chi_\rho)$ is a bijection between equivalence 
classes of Heisenberg representations of $G$ and the pairs $(Z_\rho,\chi_\rho)$ such that 
\begin{enumerate}
 \item[(a)] $Z_\rho\subseteq G$ is a coabelian normal subgroup,
 \item[(b)] $\chi_\rho$ is a $G$-invariant character of $Z_\rho$,
 \item[(c)] $X(\hat{g_1},\hat{g_2}):=\chi_\rho(g_1g_2g_1^{-1}g_2^{-1})$ is a nondegenerate 
 \textbf{alternating character} on $G/Z_\rho$ where $\hat{g_1},\hat{g_2}\in G/Z_{\rho}$ and their 
 corresponding lifts $g_1,g_2\in G$.
\end{enumerate}
\end{prop}
For pairs $(Z_\rho,\chi_\rho)$ with the properties $(a)-(c)$, the corresponding Heisenberg representation $\rho$ is determined 
by the identity (cf. \cite{SAB2}, p. 30):
\begin{equation}\label{eqn 322}
 \sqrt{[G:Z_\rho]}\cdot\rho=\mathrm{Ind}_{Z_\rho}^{G}\chi_\rho.
\end{equation}


Let
$C^1G=G$, $C^{i+1}G=[C^iG,G]$ denote the 
descending central series of $G$. Now assume that every projective representation of $A$ lifts to an ordinary representation 
of $G$. Then by I. Schur's results (cf. \cite{CR}, p. 361, Theorem 53.7) we have (cf. \cite{Z5}, p. 124, Theorem 2):
\begin{enumerate}
 \item Let $A\wedge_\bbZ A$ denote the alternating square of the $\bbZ$-module $A$. The commutator map 
 \begin{equation}\label{eqn 2.6.3}
  A\wedge_\bbZ A\cong C^2G/C^3G, \hspace{.3cm} a\wedge b\mapsto [\hat{a},\hat{b}]
 \end{equation}
is an isomorphism.
\item The map $\rho\to X_\rho\in\rm{Alt}(A)$ from Heisenberg representations to alternating characters on $A$ is 
surjective. 
\end{enumerate}

\begin{rem}\label{Remark 3.2}
 Let $\chi_\rho$ be a character of $Z_\rho$. All extensions $\chi_H\supset\chi_\rho$ are conjugate with respect  
to $G/H$. This can be easily seen, since we know $\chi_H\supset\chi_\rho$ and $\chi_{H}^{g}(h)=\chi_{H}(ghg^{-1})$. If we take 
$z\in Z_\rho$,
then we obtain
\begin{center}
 $\chi_{H}^{g}(z)=\chi_{H}(gzg^{-1})=\chi_{\rho}(gzg^{-1})=\chi_{\rho}(gzg^{-1}z^{-1}z)$\\
 $=\chi_\rho([g,z]z)=X(g,z)\cdot\chi_\rho(z)=\chi_\rho(z)$,
\end{center}
since $Z_\rho$ is a normal subgroup of $G$ and the radical of $X$ 
(i.e., $X(g,z)=\chi_\rho([g,z])=1$ for all $z\in Z_\rho$ and 
$g\in G$).
Therefore, $\chi_{H}^{g}$ are extensions of $\chi_\rho$ for all $g\in G/H$. It can also be seen that the conjugates $\chi_{H}^{g}$ are 
all different, because $\chi_{H}^{g_1}=\chi_{H}^{g_2}$ is the same as $\chi_{H}^{g_1g_{2}^{-1}}=\chi_H$.
So it is enough to see that $\chi_{H}^{g-1}\not\equiv 1$ if $g\neq1\in G/H$. But
\begin{center}
 $\chi_{H}^{g-1}(h)=\chi_\rho(ghg^{-1}h^{-1})=X(g,h)$,
\end{center}
and therefore $\chi_{H}^{g-1}\equiv 1$ on $H$ implies $g\in H^{\bot}=H$, where $``\bot"$ denotes the 
orthogonal complement with respect to $X$. Then for a given one extension $\chi_H$ of $\chi_\rho$
all other extensions are of the form $\chi_{H}^{g}$ for $g\in G/H$.

\end{rem}
\begin{rem}\label{Remark 2.10}
 Let $\rho=(Z,\chi_\rho)$ be a Heisenberg representation of $G$. Then from the definition of Heisenberg representation 
we have 
$$[[G,G], G]\subseteq \rm{Ker}(\rho).$$
Now let $\overline{G}:=G/\rm{Ker}(\rho)$. Then we obtain
$$[\overline{G},\overline{G}]=[G/\rm{Ker}(\rho),G/\rm{Ker}(\rho)]=[G,G]\cdot\rm{Ker}(\rho)/\rm{Ker}(\rho)=[G,G]/[G,G]\cap\rm{Ker}(\rho).$$
Since $[[G,G],G]\subseteq\rm{Ker}(\rho)$, then  $[x,g]\in\rm{Ker}(\rho)$ for all $x\in [G,G]$ and $g\in G$.
Hence we obtain
\begin{center}
 $[[\overline{G},\overline{G}],\overline{G}]=[[G,G]/[G,G]\cap \rm{Ker}(\rho), G/\rm{Ker}(\rho)]\subseteq\rm{Ker}(\rho)$,
\end{center}
This shows that $\overline{G}$ is a two-step nilpotent group.
\end{rem}

\section{\textbf{Arithmetic description of Heisenberg representations}}

In Section 2.5, we see the notion of Heisenberg representations of a (pro-)finite group.
These Heisenberg representations have arithmetic structure due to E.-W. Zink (cf. \cite{Z2}, \cite{Z4}, \cite{Z5}).
For this article we need to describe the arithmetic structure of Heisenberg representations.

Let $F/\bbQ_p$ be a local field, and $\overline{F}$ be an algebraic closure of $F$. Denote $G_F=\rm{Gal}(\overline{F}/F)$ the 
absolute Galois group for $\overline{F}/F$. We know that (cf. \cite{HK2}, p. 197) each representation $\rho:G_F\to GL(n,\bbC)$ corresponds 
to a projective 
representation $\overline{\rho}:G_F\to GL(n,\bbC)\to PGL(n,\bbC)$. On the other hand, each projective representation 
$\overline{\rho}:G_F\to PGL(n,\bbC)$ can be lifted to a representation $\rho:G_F\to GL(n,\bbC)$.
Let $A_F=G_{F}^{ab}$ be the factor commutator group of $G_F$. Define 
\begin{center}
 $FF^\times:=\varprojlim(F^\times/N\wedge F^\times/N)$
\end{center}
where $N$ runs over all open subgroups of finite index in $F^\times$. Denote by $\rm{Alt}(F^\times)$ as the set of 
all alternating characters $X:F^\times\times F^\times\to\bbC^\times$ such that $[F^\times:\rm{Rad}(X)]<\infty$. Then the local 
reciprocity map gives an isomorphism between $A_F$ and the profinite completion of $F^\times$, and induces a natural bijection 
\begin{equation}
 \rm{PI}(A_F)\xrightarrow{\sim}\rm{Alt}(F^\times),
\end{equation}
where $\rm{PI}(A_F)$ is the set of isomorphism classes of projective irreducible representations of $A_F$.
By using class field theory from the commutator map (\ref{eqn 2.6.3}) (cf. p. 125 of \cite{Z5}) we obtain 
\begin{equation}\label{eqn 5.1.2}
 c:FF^\times\cong [G_F,G_F]/[[G_F,G_F], G_F].
\end{equation}
 
Let $K/F$ be an abelian extension corresponding to the norm subgroup $N\subset F^\times$ and if $W_{K/F}$ denotes the relative Weil 
group, the commutator map for $W_{K/F}$ induces an isomorphism (cf. p. 128 of \cite{Z5}):
\begin{equation}\label{eqn 5.1.3}
 c: F^\times/N\wedge F^\times/N\to K_{F}^{\times}/I_{F}K^\times,
\end{equation}
where 
\begin{center}
 $K_{F}^{\times}:=\{x\in K^\times|\quad N_{K/F}(x)=1\}$, i.e., the norm-1-subgroup of $K^\times$,\\
 $I_FK^\times:=\{x^{1-\sigma}|\quad x\in K^{\times}, \sigma\in \rm{Gal}(K/F)\}<K_{F}^{\times}$, the augmentation with respect to $K/F$. 
\end{center}
Taking the projective limit over all abelian extensions $K/F$ the isomorphisms (\ref{eqn 5.1.3}) induce:
\begin{equation}\label{eqn 5.1.4}
 c:FF^\times\cong \varprojlim K_{F}^{\times}/I_FK^\times,
\end{equation}
where the limit on the right side refers to norm maps. This gives an arithmetic description of Heisenberg representations of the 
group $G_F$.

\begin{thm}[Zink, \cite{Z2}, p. 301, Corollary 1.2]\label{Theorem 5.1.1}
 The set of Heisenberg representations $\rho$ of $G_F$ is in bijective correspondence with the set of all pairs $(X_\rho,\chi_\rho)$
 such that:
 \begin{enumerate}
  \item $X_\rho$ is a character of $FF^\times$,
  \item $\chi_\rho$ is a character of $K^{\times}/I_FK^\times$, where the abelian extension $K/F$ corresponds to the radical 
  $N\subset F^\times$ of $X_\rho$, and 
  \item via (\ref{eqn 5.1.3}) the alternating character $X_\rho$ corresponds to the restriction of $\chi_\rho$ to $K_{F}^{\times}$.
 \end{enumerate}

\end{thm}
 Given a pair $(X,\chi)$, we can construct the Heisenberg representation $\rho$ by induction from $G_K:=\rm{Gal}(\overline{F}/K)$ to 
 $G_F$:
 \begin{equation}\label{eqn 5.1.5}
 \sqrt{[F^\times:N]}\cdot\rho=\rm{Ind}_{K/F}(\chi),
 \end{equation}
where $N$ and $K$ are as in (2) of the above Theorem \ref{Theorem 5.1.1}
and where the induction of $\chi$ (to be considered as a character of $G_K$ by class
field theory) produces a multiple of $\rho$. From 
$[F^\times:N]=[K:F]$ we obtain the {\bf dimension formula:}
\begin{equation}\label{eqn dimension formula}
 \rm{dim}(\rho)=\sqrt{[F^\times:N]},
\end{equation}
where $N$ is the radical of $X$.

Let $K/E$ be an extension of $E$, and $\chi_K:K^\times\to\bbC^\times$ be a character of $K^\times$. In the following lemma, we 
give the conditions of the existence of characters $\chi_E\in\widehat{E^\times}$ such that $\chi_E\circ N_{K/E}=\chi_K$, 
and the solutions set 
of this $\chi_E$. 

\begin{lem}\label{Lemma 5.1.4}
Let $K/E$ be a finite extension of a field $E$, and $\chi_K: K^\times\to\bbC^\times$.  
 \begin{enumerate}
  \item[(i)] The existence of characters $\chi_E: E^\times\to\bbC^\times$ such that $\chi_E\circ N_{K/E}=\chi_K$
  is equivalent to $K_{E}^{\times}\subset\rm{Ker}(\chi_K)$.
  \item[(ii)] In case (i) is fulfilled, we have a well defined character 
  \begin{equation}
   \chi_{K/E}:=\chi_K\circ N_{K/E}^{-1}:\mathcal{N}_{K/E}\to \bbC^\times,
  \end{equation}
on the subgroup of norms $\mathcal{N}_{K/E}:=N_{K/E}(K^\times)\subset E^\times$, and the solutions $\chi_E$ such that 
$\chi_E\circ N_{K/E}=\chi_K$ are precisely the extensions of $\chi_{K/E}$ from $\mathcal{N}_{K/E}$ to a character of 
$E^\times$.
 \end{enumerate}
\end{lem}
\begin{proof}
{\bf (i)}
Suppose that an equation $\chi_K=\chi_E\circ N_{K/E}$ holds.
Let $x\in K_{E}^{\times}$, hence $N_{K/E}(x)=1$. Then 
$$\chi_K(x)=\chi_E\circ N_{K/E}(x)=\chi_E(1)=1.$$
So $x\in\rm{Ker}(\chi_K)$, and hence $K_{E}^{\times}\subset \rm{Ker}(\chi_K)$.

Conversely assume that $K_{E}^{\times}\subset\rm{Ker}(\chi_K)$. 
 Then $\chi_K$ is actually a character of $K^\times/K_{E}^{\times}$. Again we have
 $K^\times/K_{E}^{\times}\cong \mathcal{N}_{K/E}\subset E^\times$, 
 hence $\widehat{K^\times/K_{E}^{\times}}\cong \widehat{\mathcal{N}_{K/E}}$.
 Now suppose that $\chi_K$ corresponds to the character $\chi_{K/E}$ of $\mathcal{N}_{K/E}$. Hence 
 we can write $\chi_K\circ N_{K/F}^{-1}=\chi_{K/E}$. Thus 
 the character $\chi_{K/E}:\mathcal{N}_{K/E}\to\bbC^\times$
 is well defined. Since $E^\times$ is an abelian group and $\mathcal{N}_{K/E}\subset E^\times$ is a subgroup of finite index
 (by class field theory) $[K:E]$,
 we can extend $\chi_{K/E}$ to $E^\times$, and $\chi_K$ is of the form $\chi_K=\chi_E\circ N_{K/E}$ 
 with $\chi_E|_{\cN_{K/E}}=\chi_{K/E}$.\\
 {\bf (ii)}
 If condition (i) is satisfied, then this part is obvious. 
 If $\chi_E$ is a solution of $\chi_K=\chi_E\circ N_{K/E}$, 
 with $\chi_{K/E}:=\chi_K\circ N_{K/E}^{-1}:\mathcal{N}_{K/E}\to\bbC^\times$, then certainly $\chi_E$ is an extension of 
 the character $\chi_{K/E}$. 
 
 Conversely, if $\chi_E$ extends $\chi_{K/E}$, then it is a solution of $\chi_K=\chi_E\circ N_{K/E}$ with 
 $\chi_K\circ N_{K/E}^{-1}=\chi_{K/E}:\mathcal{N}_{K/E}\to\bbC^\times$.
\end{proof}

\begin{rem}
Now take Heisenberg representation $\rho=\rho(X,\chi_K)$ of $G_F$. Let $E/F$ be any extension corresponding to a maximal 
isotropic for $X$. In this Heisenberg setting, from Theorem \ref{Theorem 5.1.1}(2), we know $\chi_K$ is a character of 
$K^\times/I_FK^\times$, and from the first commutative diagram on p. 302 of \cite{Z2} we have 
$N_{K/E}:K_F^\times/I_FK^\times\to E_F^\times/I_F\cN_{K/E}$. Thus in the Heisenberg setting,
 we have more information than Lemma \ref{Lemma 5.1.4}(i), that $\chi_K$ is a character of 
 \begin{equation}
  K^\times/K_{E}^{\times}I_FK^\times\xrightarrow{N_{K/E}}\mathcal{N}_{K/E}/I_F\mathcal{N}_{K/E}\subset E^\times/I_F\mathcal{N}_{K/E},
 \end{equation}
and therefore $\chi_{K/F}$ is actually a character of $\mathcal{N}_{K/E}/I_F\mathcal{N}_{K/E}$, or in other words, it is a 
$\rm{Gal}(E/F)$-invariant character of the $\rm{Gal}(E/F)$-module $\mathcal{N}_{K/E}\subset E^\times$. And if $\chi_E$ is one of 
the solution of Lemma \ref{Lemma 5.1.4}(ii), then the complete solutions is the set $\{\chi_E^\sigma\,|\,\sigma\in \rm{Gal}(E/F)\}$.

{\bf 
We know that $W(\chi_E,\psi\circ\rm{Tr}_{K/E})$ has the same value for all solutions $\chi_E$ of $\chi_E\circ N_{K/E}=\chi_K$,
which means for all $\chi_E$ which extend the character $\chi_{K/E}$}.

Moreover, from the above Lemma \ref{Lemma 5.1.4}, we also can see that $\chi_E|_{\mathcal{N}_{K/E}}=\chi_{K}\circ N_{K/E}^{-1}$.

\end{rem}

Let $\rho=\rho(X,\chi_K)$ be a Heisenberg representation of $G_F$. Let $E/F$ be any extension corresponding to a maximal 
isotropic for $X$. Then by using the above Lemma \ref{Lemma 5.1.4}, we have the following lemma. 

\begin{lem}
  Let $\rho=\rho(Z,\chi_\rho)=\rho(\rm{Gal}(L/K),\chi_K)$ be a Heisenberg representation of a finite local Galois group 
  $G=\rm{Gal}(L/F)$, where $F$ is a non-archimedean local field. Let $H=\rm{Gal}(L/E)$ be a maximal isotropic for 
  $\rho$. Then we obtain
  \begin{equation}
   \rho=\rm{Ind}_{E/F}(\chi_{E}^{\sigma})\quad\text{for all $\sigma\in\rm{Gal}(E/F)$},
  \end{equation}
 where $\chi_E:E^\times/I_F\cN_{K/E}\to\bbC^\times$ with $\chi_K=\chi_E\circ N_{K/E}$.\\
 Moreover, for a fixed base field $E$ of a maximal isotropic for $\rho$, this construction of 
 $\rho$ is independent of the choice of this character $\chi_E$.
 \end{lem}
\begin{proof}

From the group theoretical construction of Heisenberg representation (cf. see Section 2.6), we can write 
\begin{equation}
 \rho=\rm{Ind}_{H}^{G}(\chi_{H}^{g}), \quad\text{for all $g\in G/H$},
\end{equation}
where $\chi_H:H\to\bbC^\times$ is an extension of $\chi_\rho$. From Remark \ref{Remark 3.2} we know that all extensions of 
character $\chi_\rho$ are conjugate with respect to $G/H$, and they are different. If we fix $H$, then $\rho$ is independent
of the choice of character $\chi_H$. For every extension of $\chi_\rho$ we will have same $\rho$. The assertion of the lemma
is the arithmetic expression of this group theoretical facts, and which we will prove in the following.

By the given conditions,
 $L/F$ is a finite Galois extension of the local field $F$ and $G=\rm{Gal}(L/F)$, and 
 $H=\mathrm{Gal}(L/E)$, $Z=\mathrm{Gal}(L/K)$ and $\{1\}=\mathrm{Gal}(L/L)$.
Then by class
field theory, equation (\ref{eqn 5.1.3}), and the condition $X:=\chi_K\circ [-,-]$, 
$\chi_\rho$ identifies with a character 
\begin{center}
 $\chi_K: K^\times/I_FK^\times\to\mathbb{C}^\times$.
\end{center}
Moreover, for the Heisenberg representations we also have the following commutative diagram 

\begin{equation}
\begin{CD}
K^\times_E/I_EK^\times                    @>inclusion>>                         K^\times_F/I_FK^\times\\
@AAcA                                                                  @AAcA \\
E^\times/\cN_{K/E} \wedge E^\times/\cN_{K/E}  @>{N_{E/F}\wedge N_{E/F}}>>  F^\times/\cN_{K/F}\wedge  F^\times/\cN_{K/F}
\end{CD}
\end{equation}
where $N_{E/F}\wedge N_{E/F}(a\wedge b)=N_{E/F}(a)\wedge N_{E/F}(b)$ for all $a,b\in E^\times$, and
the vertical isomorphisms in upward direction are given as the
commutator maps (cf. equation (\ref{eqn 5.1.3})) in the Weil groups $W_{K/E}/I_EK^\times$ and
$W_{K/F}/I_FK^\times$ respectively. 
Under the right vertical $\chi_K$ corresponds (cf. Theorem \ref{Theorem 5.1.1}(3)) to the alternating
character $X$ which is trivial on $N_{E/F}\wedge N_{E/F},$ because $H$
corresponding to $E^\times$ is isotropic.
The commutative diagram  now shows that $\chi_K$ must be trivial on the
image of the upper horizontal, i.e., $\chi_K$ is trivial on the subgroups $K_{E}^{\times}$ for all maximal isotropic $E$. 
Hence $\chi_K$ is actually a character of $K^\times/K_{E}^{\times}$. 

Then from Lemma \ref{Lemma 5.1.4} we can say that there exists a character $\chi_E:E^\times/I_F\cN_{K/E}\to \bbC^\times$ such that 
$\chi_K=\chi_E\circ N_{K/E}$. And this $\chi_E$ is determined by the character $\chi_H$.
For $\sigma\in G/H=\mathrm{Gal}(E/F)$ we have $\chi_{E}^{\sigma}\circ N_{K/E}=\chi_E\circ N_{K/E}=\chi_K$ because 
$\chi_{E}^{\sigma-1}\circ N_{K/E}\equiv 1$, because $\chi_E$ is trivial on $I_F\mathcal{N}_{K/E}$.

Therefore instead of $\rho=\mathrm{Ind}_{H}^{G}(\chi_{H}^{g})$ for all $g\in G/H$, we obtain
\begin{center}
 $\rho=\rm{Ind}_{E/F}(\chi_{E}^{\sigma})$, for all $\sigma\in\rm{Gal}(E/F)$,
\end{center}
independently of the choice of $\chi_E$.

\end{proof}

\begin{rem}
Moreover we have the exact sequence
\begin{align}\label{sequence 5.1.2}
 K^\times/I_FK^\times\xrightarrow{N_{K/E}} E^\times/I_F\mathcal{N}_{K/E}\xrightarrow{N_{E/F}} F^\times/\mathcal{N}_{K/F},
\end{align}
which is only exact in the middle term. For the dual groups this gives
\begin{align}\label{sequence 5.1.3}
 \widehat{K^\times/I_FK^\times}\xleftarrow{N_{K/E}^{*}} \widehat{E^\times/I_F\mathcal{N}_{K/E}}
 \xleftarrow{N_{E/F}^{*}} \widehat{F^\times/\mathcal{N}_{K/F}}.
\end{align}
But $N_{K/E}^{*}(\chi_{E}^{\sigma-1})=\chi_{E}^{\sigma-1}\circ N_{K/E}\equiv 1$, and therefore the exactness of sequence 
(\ref{sequence 5.1.3}) yields
\begin{equation}
\chi_{E}^{\sigma-1}=\chi_F\circ N_{E/F}, \quad\text{ for some $\chi_F\in\widehat{F^\times/\mathcal{N}_{K/F}}$},
\end{equation}
\end{rem}

For our (arithmetic) determinant computation of Heisenberg representation $\rho$ of $G_F$, we need the following lemma regarding 
transfer map.

\begin{lem}\label{Lemma transfer Heisenberg}
 Let $\rho=\rho(Z,\chi_\rho)$ be a Heisenberg representation of a group $G$ and assume that $H/Z\subset G/Z$ is a maximal
 isotropic for $\rho$. Then transfer map $T_{(G/Z)/(H/Z)}\equiv1$ is the trivial map.
\end{lem}
\begin{proof}
 In general, if $H$ is a central subgroup\footnote{A subgroup of a group which lies inside the center of the group, i.e., 
 a subgroup $H$ of $G$ is central if $H\subseteq Z(G)$.} of finite index $n=[G:H]$ of a group $G$, then by Theorem 5.6 on p. 154 of \cite{MI} we have 
 $T_{G/H}(g)=g^n$. If $G$ is abelian, then center $Z(G)=G$. Hence every subgroup of $G$ is central subgroup. Now if we take $G$ as 
 an abelian group and $H$ is a subgroup of finite index, then we can write $T_{G/H}(g)=g^{[G:H]}$.
 
 Now we come to the Heisenberg setting. We know that $G/Z$ is abelian, hence $H/Z\subset G/Z$ is a central subgroup.
 Then we have $T_{(G/Z)/(H/Z)}(g)=g^{[G/Z:H/Z]}=g^d$, where $d$ is the dimension of $\rho$.
 For the Heisenberg setting, we also know (cf. Lemma 3.3 on p. 8 of \cite{SAB3}) that $G^d\subseteq Z$, hence $g^d\in Z$. This implies 
 $$T_{(G/Z)/(H/Z)}(g)=g^d=1,\quad\text{the identity in $H/Z$},$$
 for all $g\in G$, hence $T_{(G/Z)/(H/Z)}\equiv1$ is a trivial map.
\end{proof}

By using the above Lemma \ref{Lemma 5.1.4} and 
Lemma \ref{Lemma transfer Heisenberg}, in the following, we give the 
arithmetic description of the determinant of Heisenberg representations.

\begin{prop}\label{Proposition arithmetic form of determinant}
 Let $\rho=\rho(Z,\chi_\rho)=\rho(G_K,\chi_K)$ be a Heisenberg representation of the absolute Galois group $G_F$.
 Let $E$ be a base field of a maximal isotropic for $\rho$. Then $F^\times\subseteq \cN_{K/E}$, and
 \begin{equation}\label{eqn 5.1.12}
  \det(\rho)(x)=\Delta_{E/F}(x)\cdot\chi_K\circ N_{K/E}^{-1}(x)\quad \text{for all $x\in F^\times$},
 \end{equation}
where, for all $x\in F^\times$,
\begin{equation}\label{eqn 5.1.13}
 \Delta_{E/F}(x)=\begin{cases}
                  1 & \text{when $\rm{rk}_2(\rm{Gal}(E/F))\ne 1$}\\
                  \omega_{E'/F}(x) & \text{when $\rm{rk}_2(\rm{Gal}(E/F))= 1$},
                 \end{cases}
\end{equation}
where $E'/F$ is a uniquely determined quadratic subextension in $E/F$, and $\omega_{E'/F}$ is the character of $F^\times$ which 
corresponds to $E'/F$ by class field theory.
\end{prop}

\begin{proof}
 From the given condition, we can write $G/Z=\rm{Gal}(K/F)\supset H/Z=\rm{Gal}(K/E)$. Here both $G/Z$ and $H/Z$ are abelian, then from 
 class field theory we have the following commutative diagram
 \begin{equation}\label{diagram 5.1.13}
  \begin{CD}
  F^\times/\cN_{K/F}   @>inclusion>>  E^\times/\cN_{K/E}\\
  @VV\theta_{K/F}V                      @VV\theta_{K/E}V\\
  \rm{Gal}(K/F) @>T_{(G/Z)/(H/Z)}>> \rm{Gal}(K/E)
 \end{CD}
 \end{equation}
Here $\theta_{K/F}$, $\theta_{K/F}$ are the isomorphism (Artin reciprocity) maps and $T_{(G/Z)/(H/Z)}$ 
is transfer map. From Lemma \ref{Lemma transfer Heisenberg}, we have $T_{(G/Z)/(H/Z)}\equiv1$. Therefore from the above 
diagram (\ref{diagram 5.1.13}) we can say $F^\times\subseteq\cN_{K/E}$, i.e., all elements
\footnote{This condition $F^\times\subseteq\cN_{K/E}$ implies that for every $x\in F^\times$ must have a preimage under the 
$N_{K/E}$, but the preimage is not unique.} from the base field $F$ are norms with 
respect to the extension $K/E$.

Now identify $\chi_\rho=\chi_K:K^\times/I_FK^\times\to\bbC^\times$. Then the map 
$$x\in F^\times\mapsto\chi_K\circ N_{K/E}^{-1}(x)$$
is well-defined character of $F^\times$.

Now by Gallagher's Theorem (cf. \cite{GK}, Theorem $30.1.6$) (arithmetic side) we can write for all 
$x\in F^\times$,
\begin{equation}
 \det(\rho)(x)=\Delta_{E/F}(x)\cdot\chi_E(x)=\Delta_{E/F}(x)\cdot\chi_K(N_{K/E}^{-1}(x)),
\end{equation}
since $F^\times\subseteq\cN_{K/E}$, and $\chi_E|_{\cN_{K/E}}=\chi_K\circ N_{K/E}^{-1}$.

Furthermore,
since $E/F$ is an abelian extension, $\rm{Gal}(E/F)\cong\widehat{\rm{Gal}(E/F)}$, and 
from Miller's Theorem (cf. \cite{PC}, Theorem 6), we can write
\begin{align*}
 \Delta_{E/F}
 &=\det(\rm{Ind}_{E/F}(1))\\
 &=\det(\sum_{\chi\in\widehat{\rm{Gal}(E/F)}}\chi)\\
 &=\prod_{\chi\in\widehat{\rm{Gal}(E/F)}}\chi\\
 &=\begin{cases}
                  1 & \text{when $2$-rank $\rm{rk}_2(\rm{Gal}(E/F))\ne 1$}\\
                  \omega_{E'/F}(x) & \text{when $2$-rank $\rm{rk}_2(\rm{Gal}(E/F))= 1$},
                 \end{cases}
\end{align*}
where $E'/F$ is a uniquely determined quadratic subextension in $E/F$, and $\omega_{E'/F}$ is the character of $F^\times$ which 
corresponds to $E'/F$ by class field theory.

\end{proof}




\subsection{{\bf Heisenberg representations of $G_F$ of dimensions prime to $p$}}

Let $F/\bbQ_p$ be a non-archimedean local field, and $G_F$ be the absolute Galois group of $F$. In this subsection we construct all 
Heisenberg representations of $G_F$ of dimensions prime to $p$. Studying the construction of this type (i.e., dimension prime to $p$)
Heisenberg representations are important for our next section.

\begin{dfn}[{\bf U-isotropic}]\label{Definition U-isotropic}
Let $F$ be a non-archimedean local field.
 Let $X:FF^\times\to \bbC^\times$ be an alternating character with the property 
 $$X(\varepsilon_1,\varepsilon_2)=1,\qquad \text{for all $\varepsilon_1,\varepsilon_2\in U_F$}.$$
 In other words, $X$ is a character of $FF^\times/U_F\wedge U_F$. Then $X$ is said to be the U-isotropic. 
 These $X$ are easy to classify:
\end{dfn}

\begin{lem}\label{Lemma U-isotropic}
 Fix a uniformizer $\pi_F$ and write $U:=U_F$. Then we obtain an isomorphism 
 $$\widehat{U}\cong \widehat{FF^\times/U\wedge U}, \quad \eta\mapsto X_\eta,\quad \eta_X\leftarrow X$$
 between characters of $U$ and $U$-isotropic alternating characters as follows:
 \begin{equation}\label{eqn 5.1.25}
  X_\eta(\pi_F^a\varepsilon_1,\pi_F^b\varepsilon_2):=\eta(\varepsilon_1)^b\cdot\eta(\varepsilon_2)^{-a},\quad
  \eta_X(\varepsilon):=X(\varepsilon,\pi_F),
 \end{equation}
 where $a,b\in\bbZ$, $\varepsilon,\varepsilon_1,\varepsilon_2\in U$, and $\eta:U\to\bbC^\times$.
 Then 
 $$\rm{Rad}(X_\eta)=<\pi_F^{\#\eta}>\times\rm{Ker}(\eta)=<(\pi_F\varepsilon)^{\#\eta}>\times\rm{Ker}(\eta),$$
 does not depend on the choice of $\pi_F$, where  $\#\eta$ is the order of the character $\eta$, hence 
 $$F^\times/\rm{Rad}(X_\eta)\cong <\pi_F>/<\pi_F^{\#\eta}>\times U/\rm{Ker}(\eta)\cong \bbZ_{\#\eta}\times\bbZ_{\#\eta}.$$
 Therefore all Heisenberg representations of type $\rho=\rho(X_\eta,\chi)$ have dimension $\rm{dim}(\rho)=\#\eta$.
\end{lem}

\begin{proof}
To prove $\widehat{U}\cong \widehat{FF^\times/U\wedge U}$, we have to show that $\eta_{X_\eta}=\eta$ and $X_{\eta_X}=X_\eta$, and that 
the inverse map $X\mapsto \eta_X$ does not depend on the choice of $\pi_F$.

From the above definition of $\eta_X$, we can write:
\begin{align*}
 \eta_{X_\eta}(\varepsilon)
 &=X_\eta(\epsilon,\pi_F)=\eta(\varepsilon)^{1}\cdot \eta(1)^0=\eta(\varepsilon), 
\end{align*}
for all $\varepsilon\in U$, hence $\eta_{X_\eta}=\eta$.

Similarly, from the above definition of $X$, we have:
\begin{align*}
 X_{\eta_X}(\pi_F^a\varepsilon_1,\pi_F^b\varepsilon)
 &=\eta_X(\varepsilon_1)^b\cdot\eta_X(\varepsilon_2)^{-a}=X(\varepsilon_1,\pi_F)^b\cdot X(\varepsilon_2,\pi_F)^{-a}\\
 &=X(\varepsilon_1,\pi_F)^b\cdot X(\pi_F,\varepsilon_2)^{a}=X(\varepsilon_1,\pi_F^b)\cdot X(\pi_F^a,\varepsilon_2)\\
 &=X(\pi_F^a\varepsilon_1,\pi_F^b\varepsilon).
\end{align*}
This shows that $X_{\eta_X}=X$.

Now we choose a uniformizer $\pi_F\varepsilon$, where $\varepsilon\in U$, instead of choosing $\pi_F$.
Then we can write 
\begin{align*}
 X_\eta((\pi_F\varepsilon)^a\varepsilon_1,(\pi_F\varepsilon)^b\varepsilon_2)
 &=X_\eta(\pi_F^a(\varepsilon^a\varepsilon_1),\pi_F^b(\varepsilon^b\varepsilon_2))\\
 &=\eta(\varepsilon^a\varepsilon_1)^b\cdot \eta(\varepsilon^b\varepsilon_2)^{-a}\\
 &=\eta(\varepsilon_1)^b\cdot \eta(\varepsilon_2)^{-a}\cdot \eta(\varepsilon^{ab-ab})\\
 &=\eta(\varepsilon_1)^b\cdot\eta(\varepsilon_2)^{-a}=X(\pi_F^a\varepsilon_1,\pi_F^b\varepsilon_2).
\end{align*}
This shows that $X_\eta$ does not depend on the choice of the uniformizer $\pi_F$. Similarly since 
$\eta_X(\varepsilon):=X(\varepsilon,\pi_F)$, it is clear that $\eta_X$ is also does not depend on the choice of the 
uniformizer $\pi_F$.

By the definition of the radical of $X_\eta$, we have:
 $$\rm{Rad}(X_\eta)=
 \{\pi_F^a\varepsilon\in F^\times\,|\; X_\eta(\pi_F^{a}\varepsilon,\pi_F^{b}\varepsilon')= 
 \eta(\varepsilon)^{b}\cdot \eta(\varepsilon')^{-a}=1\},$$
 for all $b\in \bbZ$, and $\varepsilon'\in U$. 
 
 Now if we fix a uniformizer $\pi_F\varepsilon'',$ where $\varepsilon''\in U$ instead of $\pi_F$, we can write:
 $$\rm{Rad}(X_\eta)=
 \{(\pi_F\varepsilon'')^a\varepsilon\in F^\times\,|\; X_\eta((\pi_F\varepsilon'')^{a}\varepsilon,(\pi_F\varepsilon'')^{b}\varepsilon')= 
 \eta(\varepsilon''^a\varepsilon)^{b}\cdot \eta(\varepsilon''^b\varepsilon')^{-a}=\eta(\varepsilon)^b\cdot\eta(\varepsilon')^{-a}=1\},$$

 This gives $\rm{Rad}(X_\eta)=<\pi_F^{\#\eta}>\times\rm{Ker}(\eta)=<(\pi_F\varepsilon)^{\#\eta}>\times\rm{Ker}(\eta)$, hence 
 $$F^\times/\rm{Rad}(X_\eta)\cong <\pi_F>/<\pi_F^{\#\eta}>\times U/\rm{Ker}(\eta)\cong \bbZ_{\#\eta}\times\bbZ_{\#\eta}.$$
 
 Then all Heisenberg representations of type $\rho=\rho(X_\eta,\chi)$ have dimension
 $$\rm{dim}(\rho)=\sqrt{[F^\times:\rm{Rad}(X_\eta)]}=\#\eta.$$
 
\end{proof}

From the above Lemma \ref{Lemma U-isotropic} we know that the dimension of a U-isotropic Heisenberg representation 
$\rho=\rho(X_\eta,\chi)$ of $G_F$ is $\rm{dim}(\rho)=\#\eta$, and $F^\times/\rm{Rad}(X_\eta)\cong \bbZ_{\#\eta}\times\bbZ_{\#\eta}$,
a direct product of two cyclic (bicyclic) groups of the same order $\#\eta$. In general, if $A=\bbZ_m\times\bbZ_m$ is a bicyclic
group of order $m^2$, then by the following lemma we can compute total number of elements of order $m$ in $A$, and 
number of cyclic complementary subgroup of a fixed cyclic subgroup of order $m$. 

\begin{lem}\label{Lemma on bicyclic abelian groups}
 Let $A\cong \bbZ_m\times\bbZ_m$ be a bicyclic abelian group of order $m^2$. Then:
 \begin{enumerate}
  \item Then number $\psi(m)$ of cyclic subgroups $B\subset A$ of order $m$ is a multiplicative arithmetic function 
  (i.e., $\psi(mn)=\psi(m)\psi(n)$ if $gcd(m,n)=1$).
  \item Explicitly we have 
  \begin{equation}
   \psi(m)=m\cdot\prod_{p|m}(1+\frac{1}{p}).
  \end{equation}
And the number of elements of order $m$ in $A$ is:
\begin{equation}
 \varphi(m)\cdot\psi(m)=m^2\cdot\prod_{p|m}(1-\frac{1}{p^2}).
\end{equation}
Here $p$ is a prime divisor of $m$ and $\varphi(n)$ is the Euler's totient function of $n$.
\item Let $B\subset A$ be cyclic of order $m$. Then $B$ has always a complementary subgroup $B'\subset A$ such that $A=B\times B'$,
and $B'$ is again cyclic of order $m$. And for $B$ fixed, the number of all different complementary subgroups 
$B'$ is $=m$.
 \end{enumerate}
\end{lem}

\begin{proof}
To prove these assertions we need to recall the fact: If $G$ is a finite cyclic group of order $m$, then number of generators of 
$G$ is $\varphi(m)=m\prod_{p|m}(1-\frac{1}{p})$.\\  
 {\bf (1).} By the given condition $A\cong \bbZ_m\times\bbZ_m$ and $\psi(m)$ is the number of cyclic subgroup of $A$ of order $m$.
 Then it is clear that $\psi$ is an arithmetic function with $\psi(1)=1\ne 0$, hence $\psi$ is not {\bf additive}.
 Now take $m\ge 2$, and the prime factorization of $m$ is: $m=\prod_{i=1}^{k}p_{i}^{a_i}$.
 To prove this, first we should start with $m=p^n$, hence $A\cong \bbZ_{p^n}\times\bbZ_{p^n}$. Then number of subgroup of $A$ of order 
 $p^n$ is:
 $$\psi(p^n)=\frac{2\varphi(p^n)p^n-\varphi(p^n)^2}{\varphi(p^n)}=2p^n-\varphi(p^n)=p^n(2-1+\frac{1}{p})=p^n(1+\frac{1}{p}).$$
 
Now take $m=p^nq^r$, where $p,q$ are both prime with $gcd(p,q)=1$. 
 We also know that $\bbZ_{p^nq^r}\times\bbZ_{p^nq^r}\cong\bbZ_{p^n}\times\bbZ_{p^n}\times\bbZ_{q^r}\times\bbZ_{q^r}$.
 This gives $\psi(p^nq^r)=\psi(p^n)\cdot\psi(q^r)$. By the similar method we can show that 
 $\psi(m)=\prod_{i=i}^{k}\psi(p_{i}^{a_i})$, where  $m=\prod_{i=1}^{k}p_{i}^{a_i}$.
 This condition implies that $\psi$ is a multiplicative arithmetic function.\\

 {\bf (2).} Since $\psi$ is multiplicative arithmetic function, we have
 \begin{align*}
  \psi(m)
  &=\prod_{i=1}^{k}\psi(p_{i}^{a_i})=\prod_{i=1}^{k}p_{i}^{a_i}(1+\frac{1}{p_i})\quad\text{since $\psi(p^n)=p^n(1+\frac{1}{p})$},\\
  &=p_{1}^{a_1}\cdots p_{k}^{a_k}\prod_{i=1}^{k}(1+\frac{1}{p_i})=m\cdot\prod_{p|m}(1+\frac{1}{p}).
 \end{align*}
We also know that number of generator of a finite cyclic group of order $m$ is $\varphi(m)$, hence number of elements of order 
$m$ is $\varphi(m)$. Then the number of elements of order $m$ in $A$ is:
\begin{equation*}
 \varphi(m)\cdot\psi(m)=m\cdot\prod_{p|m}(1-\frac{1}{p})\cdot m\prod_{p|m}(1+\frac{1}{p})=m^2\cdot\prod_{p|m}(1-\frac{1}{p^2}).
\end{equation*}
{\bf (3).} Let $B\subset A$ be a cyclic subgroup of order $m$. Since $A$ is abelian and bicyclic of order $m^2$, $B$ has always 
a complementary subgroup $B'\subset A$ such that $A=B\times B'$, and $B'$ is again cyclic (because $A$ is cyclic, hence 
$A/B$ and $|A/B|=m$) of order $m$.

To prove the last part of (3), we start with $m=p^n$. Here $B$ is a cyclic subgroup of $A$ of order $p^n$, hence 
$B=<(a,e)>$, where $\# a=p^n$, and $e$ is the identity of $B'$. 
Since $B$ has complementary cyclic subgroup, namely $B'$, of order $p^n$. we can choose 
$B'=<(b,c)>$, where $B\cap B'=(e,e)$. This gives that $c$ is a generator of $B'$,  and $b$ could be any element in $\bbZ_{p^n}$.
Thus total number $\psi_{B'}(p^n)$ of all different complementary subgroups $B'$ is:
$$\psi_{B'}(p^n)=\frac{p^n\varphi(p^n)}{\varphi(p^n)}=p^n=m.$$
Now if we take $m=p^nq^r$, where $q$ is a different prime from $p$. Then by same method we can see that 
$\psi_{B'}(p^nq^r)=\psi_{B'}(p^n)\cdot\psi_{B'}(q^r)=p^nq^r=m$. Thus for arbitrary $m$ we can conclude that 
$\psi_{B'}(m)=m$.

\end{proof}

In the following lemma, we give an equivalent condition for U-isotropic Heisenberg representation.

\begin{lem}\label{Lemma U-equivalent}
Let $G_F$ be the absolute Galois group of a non-archimedean local field $F$.
 For a Heisenberg representation $\rho=\rho(Z,\chi_\rho)=\rho(X,\chi_K)$ the following are equivalent:
 \begin{enumerate}
  \item The alternating character $X$ is U-isotropic.
  \item Let $E/F$ be the maximal unramified subextension in $K/F$. Then $\rm{Gal}(K/E)$ is maximal isotropic for $X$.
  \item $\rho=\rm{Ind}_{E/F}(\chi_E)$ can be induced from a character $\chi_E$ of $E^\times$ (where $E$ is as in (2)).
 \end{enumerate}
\end{lem}
\begin{proof}
 This proof follows from the above Lemma \ref{Lemma U-isotropic}. \\
 First, assume that $X$ is U-isotropic, i.e., $X\in\widehat{FF^\times/U\wedge U}$. We also know that 
 $\widehat{U}\cong\widehat{FF^\times/U\wedge U}$. Then $X$ corresponds a character of $U$, namely $X\mapsto \eta_X$.
 Then from Lemma \ref{Lemma U-isotropic} we have $F^\times /\rm{Rad}(X)\cong \bbZ_{\#\eta_X}\times\bbZ_{\#\eta_X}$, i.e.,
 product of two cyclic groups of same order.
 
 Since $K/F$ is the abelian bicyclic extension which corresponds to $\rm{Rad}(X)$, we can write:
 $$\cN_{K/F}=\rm{Rad}(X),\qquad\rm{Gal}(K/F)\cong F^\times/\rm{Rad}(X).$$
Let $E/F$ be the maximal unramified subextension in $K/F$. Then $[E:F]=\#\eta_K$ because the order of 
maximal cyclic subgroup of $\rm{Gal}(K/F)$ is $\#\eta_X$. Then $f_{E/F}=\#\eta_X$, hence 
$f_{K/F}=e_{K/F}=\#\eta_X$ because $f_{K/F}\cdot e_{K/F}=[K:F]=\#\eta_X^2$ and $\rm{Gal}(K/F)$ is not cyclic group.

Now we have to prove that the extension $E/F$ corresponds to a maximal isotropic for $X$. Let $H/Z$ be a maximal isotropic for 
$X$, hence $[G_F/Z:H/Z]=\#\eta_X$, hence $H/Z=\rm{Gal}(K/E)$, i.e., the maximal unramified subextension $E/F$ in $K/F$ corresponds
to a maximal isotropic subgroup, hence 
\begin{center}
 $\rho(X,\chi_K)=\rm{Ind}_{E/F}(\chi_E)$, for $\chi_E\circ N_{K/E}=\chi_K$.
\end{center}
Finally, since $E/F$ is unramified and the extension $E$ corresponds a maximal isotropic subgroup for $X$, we have 
$U_F\subset\cN_{E/F}$, hence $U_F\subset\cN_{K/F}$ and $X|_{U\times U}=1$ because $U_F\subset F^\times\subset\cN_{K/E}$. 
This shows that $X$ is U-isotropic.
\end{proof}

\begin{cor}\label{Corollary U-isotropic}
 The U-isotropic Heisenberg representation $\rho=\rho(X_\eta,\chi)$ can never be wild because it is induced from 
 an unramified extension $E/F$, 
 but the dimension $\rm{dim}(\rho(X_\eta,\chi))=\#\eta$ can be a power of $p.$\\
The representations $\rho$ of dimension prime to p are precisely given as 
$\rho=\rho(X_\eta,\chi)$ for characters $\eta$ of $U/U^1.$
\end{cor}
\begin{proof}
 This is clear from the above lemma \ref{Lemma U-isotropic} and the fact $|U/U^1|=q_F-1$.
\end{proof}

\begin{rem}[{\bf Arithmetic description of representations $\rho(X_\eta,\chi)$:}] 
We let $K_\eta|F$ be the abelian bicyclic 
extension which corresponds to $\rm{Rad}(X_\eta):$

$$ \cN_{K_\eta/F}= \rm{Rad}(X_\eta),\qquad \rm{Gal}(K_\eta/F)\cong F^\times/\rm{Rad}(X_\eta).$$
Then we have $f_{K_\eta|F}= e_{K_\eta|F}=\#\eta$ and the maximal unramified subextension 
$E/F\subset K_\eta/F$ corresponds to a maximal isotropic subgroup, hence
$$ \rho(X_\eta,\chi) = \rm{Ind}_{E/F}(\chi_E),\quad\textrm{for}\; \chi_E\circ N_{K_\eta/E} =\chi.$$
We recall here that $\chi:K_\eta^\times/I_FK_\eta^\times\rightarrow\bbC^\times$ is a character such that
(cf. Theorem \ref{Theorem 5.1.1}(3))
$$ \chi|_{(K_\eta^\times)_F} \leftrightarrow X_\eta,\quad\textrm{with respect to}\; 
(K_\eta^\times)_F/I_FK_\eta^\times\cong F^\times/\rm{Rad}(X_\eta)\wedge F^\times/\rm{Rad}(X_\eta).$$
In particular, we see that $(K_\eta^\times)_F/I_FK_\eta^\times$ is cyclic of 
order $\#\eta$ and $\chi|_{(K_\eta^\times)_F}$ must be a faithful character of that cyclic group.
\end{rem}
In the following lemma we see the explicit description of the representation $\rho=\rho(X_\eta,\chi)$.

\begin{lem}[{\bf Explicit Lemma}]\label{Explicit Lemma}
 Let $\rho=\rho(X_\eta,\chi_K)$ be a U-isotropic Heisenberg representation of the absolute Galois group $G_F$ of a local field 
 $F/\bbQ_p$. Let $K=K_\eta$ and let $E/F$ be the maximal unramified subextension in $K/F$. Then: 
 \begin{enumerate}
  \item The norm map induces an isomorphism:
  $$N_{K/E}:K_F^\times/I_FK^\times\stackrel{\sim}{\to}I_FE^\times/I_F\cN_{K/E}.$$
  \item Let $c_{K/F}:F^\times/\rm{Rad}(X_\eta)\wedge F^\times/\rm{Rad}(X_\eta)\cong K_F^\times/I_FK^\times$ be the isomorphism
  which is induced by the commutator in the relative Weil-group $W_{K/F}$. Then for units $\varepsilon\in U_F$ we 
  explicitly have:
  $$c_{K/F}(\varepsilon\wedge\pi_F)=N_{K/E}^{-1}(N_{E/F}^{-1}(\varepsilon)^{1-\varphi_{E/F}}),$$
  where $\varphi_{E/F}$ is the Frobenius automorphism for $E/F$ and where $N^{-1}$ means to take a preimage of the norm map.
  \item The restriction $\chi_K|_{K_F^\times}$ is characterized by:
  $$\chi_K\circ c_{K/F}(\varepsilon\wedge\pi_F)=X_\eta(\varepsilon,\pi_F)=\eta(\varepsilon),$$
  for all $\varepsilon\in U_F$, where $c_{K/F}(\varepsilon\wedge\pi_F)$ is explicitly given via (2).
 \end{enumerate}

\end{lem}

\begin{proof}
 {\bf (1).} By the given conditions we have: $K=K_\eta,$ and $K/F$ is the bicyclic extension with $\rm{Rad}(X_\eta)=\cN_{K/F}$, and 
 $E/F$ is the maximal unramified subextension in $K/F$. So $K/E$ and $E/F$ both are cyclic, hence 
 $$E_F^\times=I_FE^\times,\qquad K_E^\times=I_EK^\times.$$
 From the diagram (3.6.1) on p. 41 of \cite{Z4}, we have 
 $$N_{K/E}: K_F^\times/I_FK^\times\stackrel{\sim}{\to} E_F^\times/I_F\cN_{K/E}.$$
 We also know that $E_F^\times=I_FE^\times$. Thus the norm map $N_{K/E}$ induces an isomorphism:
 $$N_{K/E}:K_F^\times/I_FK^\times\cong I_FE^\times/I_F\cN_{K/E}.$$
 {\bf (2).} By the given conditions, $c_{K/F}$ is the isomorphism which is induced by the commutator in 
 the relative Weil-group  $W_{K/F}$
 (cf. the map (\ref{eqn 5.1.3}). Here $\rm{Rad}(X_\eta)=\cN_{K/F}=:N$. Then from Proposition 1(iii) of \cite{Z5} on p. 128, we have 
 $$c_{K/F}: N\wedge F^\times/N\wedge N\stackrel{\sim}{\to} I_FK^\times/I_FK_F^\times$$
 as an isomorphism by the map:
 $$c_{K/F}(x\wedge y)=N_{K/F}^{-1}(x)^{1-\phi_F(y)},$$
 where $\phi_F(y)\in \rm{Gal}(K/F)$ for $y\in F^\times$ by class field theory.
 If $y=\pi_F$, then by class field theory (cf. \cite{JM}, p. 20, Theorem 1.1(a)), we can write 
 $\phi_F(\pi_F)|_{E}=\varphi_{E/F}$, where $\varphi_{E/F}$ is the Frobenius automorphism for $E/F$.
 
 Now we come to our special case.
Since $E/F$ is unramified, we have $U_F\subset\cN_{E/F}$, and we obtain (cf. \cite{Z4}, pp. 46-47 of Section 4.4 and 
the diagram on p. 302 of \cite{Z2}):
\begin{equation}\label{eqn explicit lemma}
 N_{K/E}\circ c_{K/F}(\varepsilon\wedge\pi_F)=N_{E/F}^{-1}(\varepsilon)^{1-\varphi_{E/F}}.
\end{equation}
We also know (see the first two lines under the upper diagram on p. 302 of \cite{Z2}) that
$E_F^\times\subseteq \cN_{K/E}$. Here 
$$N_{E/F}^{-1}(\varepsilon)^{1-\varphi_{E/F}}\in I_FE^\times/I_F\cN_{K/E}=E_F^\times/I_F\cN_{K/E},$$
because $E/F$ is cyclic, hence $E_F^\times=I_FE^\times$. Therefore from equation (\ref{eqn explicit lemma}) we can conclude:
$$c_{K/F}(\varepsilon\wedge\pi_F)=N_{K/E}^{-1}(N_{E/F}^{-1}(\varepsilon)^{1-\varphi_{E/F}}).$$
{\bf (3.)} We know that the $c_{K/F}(\varepsilon\wedge\pi_F)\in K_F^\times$ and $\chi_K:K^\times/I_FK^\times\to\bbC^\times$. 
Then we can write 
\begin{align*}
 \chi_K\circ c_{K/F}(\varepsilon\wedge\pi_F)
 &=\chi_K(N_{K/E}^{-1}(N_{E/F}^{-1}(\varepsilon)^{1-\varphi_{E/F}})\\
 &=\chi_E\circ N_{K/E}(N_{K/E}^{-1}(N_{E/F}^{-1}(\varepsilon)^{1-\varphi_{E/F}}), \quad\text{since $\chi_K=\chi_E\circ N_{K/E}$}\\
 &=\chi_E(N_{E/F}^{-1}(\varepsilon)^{1-\varphi_{E/F}})=X_\eta(\varepsilon,\pi_F)\\
 &=\eta(\varepsilon).
\end{align*}
This is true for all $\varepsilon\in U_F$. Therefore we can conclude that 
$\chi_K|_{K_F^\times}=\eta$.
\end{proof}

\begin{exm}[{\bf Explicit description of Heisenberg representations of dimension prime to $p$}]\label{Example for Heisenberg reps}

Let $F/\bbQ_p$ be a local field, and $G_F$ be the absolute Galois group of $F$.
Let $\rho=\rho(X,\chi_K)$ be a Heisenberg representation of $G_F$ of dimension $m$ prime to $p$. Then from 
Corollary \ref{Corollary U-isotropic} the alternating character $X=X_\eta$ is $U$-isotropic for a character
$\eta:U_F/U_F^1\to\bbC^\times$. Here from Lemma \ref{Lemma U-isotropic} 
we can say $m=\sqrt{[F^\times:\rm{Rad}(X_\eta)]}=\#\eta$ divides $q_F-1$.

Since $U_F^1$ is a pro-p-group and $gcd(m,p)=1$, we have $(U_F^1)^m=U_F^1\subset {F^\times}^m$, and therefore  
$$F^\times/{F^\times}^m\cong\bbZ_m\times\bbZ_m,$$
is a bicyclic group of order $m^2$. So by class field theory there is precisely one extension $K/F$ such that 
$\rm{Gal}(K/F)\cong\bbZ_m\times\bbZ_m$ and the norm group $\cN_{K/F}:=N_{K/F}(K^\times)={F^\times}^m$.

We know that $U_F/U_F^1$ is a cyclic group of order $q_F-1$, hence $\widehat{U_F/U_F^1}\cong U_F/U_F^1$. By the given condition 
$m|(q_F-1)$, hence $U_F/U_F^1$ has exactly one subgroup of order $m$. Then number of elements of order $m$ in $U_F/U_F^1$ is 
$\varphi(m)$, the Euler's $\varphi$-function of $m$.
In this setting, we have $\eta\in \widehat{U_F/U_F^1}\cong \widehat{FF^\times/U_F^1\wedge U_F^1}$ with 
$\#\eta=m$. This implies that up to $1$-dimensional character twist there are $\varphi(m)$ representations 
corresponding to $X_\eta$ where $\eta:U_F/U_F^1\to\bbC^\times$ is of order $m$.
According to Corollary 1.2 of \cite{Z2}, all dimension-m-Heisenberg 
representations of $G_F=\rm{Gal}(\overline{F}/F)$ are given as 
\begin{equation}
 \rho=\rho(X_\eta,\chi_K),\tag{1H}
\end{equation}
where $\chi_K: K^\times/ I_{F}K^\times\to\mathbb{C}^{\times}$ is a character 
such that the restriction of $\chi_K$
to the subgroup $K_{F}^{\times}$ corresponds to $X_\eta$ under the map (\ref{eqn 5.1.3}), and
\begin{equation}
 F^\times/{F^\times}^m\wedge F^\times/{F^\times}^m\cong K_{F}^{\times}/I_{F}K^\times,\tag{2H}
\end{equation}
which is given via the commutator in the relative Weil-group $W_{K/F}$ (for details arithmetic description of Heisenberg
representations of a Galois group, see \cite{Z2}, pp. 301-304).
The condition (2H) corresponds to (\ref{eqn 5.1.3}). Here the above Explicit Lemma \ref{Explicit Lemma} comes in.

Here due to our assumption both sides of (2H) are groups of order $m$.
And if one choice $\chi_K=\chi_0$ has been fixed, then all other $\chi_K$
are given as
\begin{equation}\label{eqn 4.20}
 \chi_K=(\chi_F\circ N_{K/F})\cdot\chi_0,
\end{equation}
for arbitrary characters of $F^\times$. For an optimal choice $\chi_K=\chi_0$, and order of $\chi_0$ we need the following lemma.

\begin{lem}\label{Lemma 5.3.3}
Let $K/F$ be the extension of $F/\bbQ_p$ for which $\rm{Gal}(K/F)=\bbZ_m\times\bbZ_m$. 
The $K_{F}^{\times}$ and $I_{F}K^\times$ are
as above. Then 
 the sequence 
 \begin{equation}\label{eqn 4.21}
  1\to U_{K}^{1}K_{F}^{\times}/U_{K}^{1}I_{F}K^\times\to U_K/U_{K}^{1}I_{F}K^\times\xrightarrow{N_{K/F}} U_F/U_{F}^{1}\to
  U_F/U_F\cap {F^\times}^m\to 1
 \end{equation}
is exact, and the outer terms are both of order $m$, hence inner terms are both cyclic of order $q_F-1$.
\end{lem}
\begin{proof}
 The sequence is exact because ${F^\times}^m=N_{K/F}(K^\times)$ is the group of norms, and 
 $F^\times/{F^\times}^m\cong \bbZ_m\times\bbZ_m$ implies
 that the right hand term\footnote{Since $gcd(m,p)=1$, we have 
 \begin{center}
  $U_F\cdot{F^\times}^m=(<\zeta>\times U_F^1)(<\pi_F^m>\times<\zeta^m>\times U_F^1)=<\pi_F^m>\times<\zeta>\times U_F^1$,
 \end{center}
where $\zeta$ is a $(q_F-1)$-st root of unity. 
Then 
\begin{center}
 $U_F/U_F\cap {F^\times}^m=U_F\cdot {F^\times}^m/{F^\times}^m=
 <\pi_F^m>\times<\zeta>\times U_F^1/<\pi_F^m>\times<\zeta^m>\times U_F^1\cong\bbZ_m$.
\end{center}
Hence $|U_F/U_F\cap{F^\times}^m|=m$.} is of order $m$. By our assumption the order of $K_{F}^{\times}/I_{F}K^\times$ is $m$. Now 
 we consider the exact sequence
 \begin{equation}\label{sequence 5.1.25}
  1\to U_{K}^{1}\cap K_{F}^{\times}/U_{K}^{1}\cap I_{F}K^\times\to K_{F}^{\times}/I_{F}K^\times\to 
  U_{K}^{1}K_{F}^{\times}/U_{K}^{1}I_{F}K^\times\to 1.
 \end{equation}
Since the middle term has order $m$, the left term must have order $1$, because $U_{K}^{1}$ is a pro-p-group and $gcd(m,p)=1$.
Hence the right term is also of order $m$. So the outer terms of the sequence (\ref{eqn 4.21}) have both order $m$, hence the inner 
terms must have the same order $q_F-1=[U_F:U_{F}^{1}]$, and they are cyclic, because the groups $U_F/U_{F}^{1}$ and $U_K/U_{K}^{1}$
are both cyclic.
\end{proof}

{\bf\large{We now are in a position to choose $\chi_K=\chi_0$ as follows}}: 
\begin{enumerate}
 \item we take $\chi_0$ as a character of $K^\times/U_{K}^{1}I_{F}K^\times$,
 \item we  take it on $U_{K}^{1}K_{F}^{\times}/U_{K}^{1}I_{F}K^\times$ as it is prescribed by the above 
 Explicit Lemma \ref{Explicit Lemma},
 in particular, $\chi_0$ restricted to that subgroup (which is cyclic of order $m$) will be faithful.
 \item we take it trivial on all primary components of the cyclic group $U_{K}/U_{K}^{1}I_{F}K^\times$ which are not $p_i$-primary,
 where $m=\prod_{i=1}^{n}p_i^{a_i}$.
 \item we take it trivial for a fixed prime element $\pi_K$.
\end{enumerate}

Under the above optimal choice of $\chi_0$, we have

\begin{lem}\label{Lemma 5.1.17}
Denote $\nu_p(n):=$ as the highest power of $p$ for which $p^{\nu_p(n)}|n$.
 The character $\chi_0$ must be a character of order 
 $$m_{q_F-1}:=\prod_{l|m}l^{\nu_l(q_F-1)},$$
 which we will call the $m$-primary part of $q_F-1$, so it determines a cyclic
extension $L/K$ of degree $m_{q_F-1}$ which is totally tamely ramified, and we can consider 
the Heisenberg representation $\rho=(X,\chi_0)$ of 
$G_F=\rm{Gal}(\overline{F}/F)$ is a representation of $\rm{Gal}(L/F)$, which is of order $m^2\cdot m_{q_F-1}$.
\end{lem}

\begin{proof}
By the given conditions, $m|q_F-1$. Therefore we can write
$$q_F-1=\prod_{l|m}l^{\nu_l(q_F-1)}\cdot \prod_{p|q_F-1,\; p\nmid m}p^{\nu_p(q_F-1)}=
m_{q_F-1}\cdot \prod_{p|q_F-1,\;p\nmid m}p^{\nu_p(q_F-1)},$$
where $l, p$ are prime, and $m_{q_F-1}=\prod_{l|m}l^{\nu_l(q_F-1)}$.

From the construction of $\chi_0$, $\pi_K\in\rm{Ker}(\chi_0)$, hence the order of $\chi_0$ comes from the restriction to 
$U_K$. Then the order of $\chi_0$ is $m_{q_F-1}$, because from Lemma \ref{Lemma 5.3.3}, the order of $U_K/U_{K}^{1}I_FK$ is
$q_F-1$. Since order of $\chi_0$ is $m_{q_F-1}$, by class field theory $\chi_0$ determines a cyclic 
extension $L/K$ of degree $m_{q_F-1}$, hence 
$$N_{L/K}(L^\times)=\rm{Ker}(\chi_0)=\rm{Ker}(\rho).$$
This means $G_L$ is the kernel of $\rho(X,\chi_0)$, hence $\rho(X,\chi_0)$ is actually a representation of 
$G_F/G_L\cong\rm{Gal}(L/F)$.

Since $G_L$ is normal subgroup of $G_F$, hence $L/F$ is a normal extension of degree $[L:F]=[L:K]\cdot[K:F]=m_{q_F-1}\cdot m^2$.
Thus $\rm{Gal}(L/F)$ is of order $m^2\cdot m_{q_F-1}$.

Moreover since $[L:K]=m_{q_F-1}$ and $gcd(m,p)=1$, $L/K$ is tame. By construction we have a prime 
$\pi_K\in\rm{Ker}(\chi_0)=N_{L/K}(L^\times)$, hence $L/K$ is totally ramified extension. 

\end{proof}

\begin{lem}(Here $L$, $K$, and $F$ are the same as in Lemma \ref{Lemma 5.1.17})
 Let $F^{ab}/F$ be the maximal abelian extension. Then we have 
$$L\supset L\cap F^{ab}\supset K\supset F, \quad\{1\}\subset G'\subset Z(G)\subset G=\rm{Gal}(L/F),$$
where $[L:L\cap F^{ab}]=|G|=m$ and $[L:K]=|Z(G)|=m_{q_F-1}$.
\end{lem}

\begin{proof}
 Let $F^{ab}/F$ be the maximal abelian extension. Then we have 
$$L\supset L\cap F^{ab}\supset K\supset F.$$
Here $L\cap F^{ab}/F$ is the maximal abelian in $L/F$. Then from Galois theory we can conclude 
$$\rm{Gal}(L/L\cap F^{ab})=[\rm{Gal}(L/F), \rm{Gal}(L/F)]=: G'.$$
Since $\rm{Gal}(L/F)=G_F/\rm{Ker}(\rho)$, and $[[G_F,G_F],G_F]\subseteq\rm{Ker}(\rho)$, from relation (\ref{eqn 5.1.3}) we have 
$$G'=[G_F,G_F]/\rm{Ker}(\rho)\cap [G_F,G_F]=[G_F,G_F]/[[G_F,G_F],G_F]\cong K_F^\times/I_FK^\times.$$
Again from sequence \ref{sequence 5.1.25} we have $|U_K^1K_F^\times/U_K^1 I_FK^\times|=|K_F^\times/I_FK^\times|=m$.
Hence $|G'|=m$.


From the Heisenberg property of $\rho$, we have 
$[[G_F,G_F],G_F]\subseteq\rm{Ker}(\rho)$, hence $\rm{Gal}(L/F)=G_F/\rm{Ker}(\rho)$ is a two-step nilpotent group 
(cf. Remark \ref{Remark 2.10}).
This gives $[G',G]=1$, hence $G'\subseteq Z:=Z(G)$. Thus $G/Z$ is abelian. 

Moreover, here $Z$ is the scalar group of $\rho$, hence the dimension of $\rho$ is:
$$\rm{dim}(\rho)=\sqrt{[G:Z]}=m$$
Therefore the order of $Z$ is $m_{q_F-1}$ and $Z=\rm{Gal}(L/K)$.

\end{proof}

\begin{rem}[{\bf Special case: $m=2$, hence $p\ne 2$}]

Now if we take $m=2$, hence $p\ne 2$, and choose $\chi_0$ as the above optimal choice, then we will have 
$m_{q_F-1}=2_{q_F-1}=2$-primary factor of the number $q_F-1$, and $\rm{Gal}(L/F)$ is a $2$-group of order 
$4\cdot 2_{q_F-1}$.

 When $q_F\equiv -1\pmod{4}$, $q_F$ is of the form $q_F=4l-1$, where $l\ge 1$. So we can write $q_F-1=2(2l-1)$.
Since $2l-1$ is always odd, therefore when $q_F\equiv-1\pmod{4}$, the order of $\chi_0$ is $2_{q_F-1}=2$. 
Then $\rm{Gal}(L/F)$ will be of order 8 if and only if $q_F\equiv -1\pmod{4}$, i.e., if and only
if $i\not\in F$. And if $q_F\equiv 1\pmod{4}$, then similarly,  we can write $q_F-1=4m$ for some integer $m\ge1$, hence 
$2_{q_F-1}\ge 4$. Therefore when $q_F\equiv 1\pmod{4}$, the order of $\rm{Gal(L/F)}$ will be at least $16$.

\end{rem}

\end{exm}

\subsection{{\bf Artin conductors, Swan conductors, and the dimensions of Heisenberg representations}}

\begin{dfn}[{\bf Artin and Swan conductor}]
 Let $G$ be a finite group and $R(G)$ be the complex representation ring of $G$. For any two representations 
 $\rho_1,\rho_2\in R(G)$ with characters $\chi_1,\chi_2$ respectively, we have the Schur's inner product:
 $$<\rho_1,\rho_2>_G=<\chi_1,\chi_2>_G:=\frac{1}{|G|}\sum_{g\in G}\chi_1(g)\cdot\overline{\chi_2(g)}.$$
 Let $K/F$ be a finite Galois group with Galois
 group $G:=\rm{Gal}(K/F)$. For an element $g\in G$ different from identity $1$, we define the positive integer 
 (cf. \cite{JPS}, Chapter IV, p. 62)
 $$i_G(g):=\rm{inf}\{\nu_K(x-g(x))|\; x\in O_K\}.$$
 By using this non-negative (when $g\ne 1$) integer $i_G(g)$ we define a function $a_G:G\to\bbZ$ as follows:
 \begin{center}
  $a_G(g)=-f_{K/F}\cdot i_G(g)$ when $g\ne 1$, and $a_G(1)=f_{K/F}\sum_{g\ne 1}i_G(g)$.
 \end{center}
Thus from this definition we can see that $\sum_{g\in G}a_G(g)=0$, hence $<a_G, 1_G>=0$. 
It can be proved (cf. \cite{JPS}, p. 99, Theorem 1) that the function $a_G$ is the character of a linear representation of $G$,
and that corresponding linear representation is called the {\bf Artin representation} $A_G$ of $G$.

Similarly, for a nontrivial $g\ne 1\in G$, we define (cf. \cite{VS}, p. 247)
$$s_G(g)=\rm{inf}\{\nu_K(1-g(x)x^{-1})|\;x\in K^\times\},\qquad s_G(1)=-\sum_{g\ne 1}s_G(g).$$
And we can define a function $\rm{sw}_G:G\to\bbZ$ as follows:
$$\rm{sw}_G(g)=-f_{K/F}\cdot s_G(g)$$
It can also be shown that $\rm{sw}_G$ is a character of a linear representation of $G$, and that corresponding representation
is called the {\bf Swan representation} $SW_G$ of $G$.

From \cite{JP}, p. 160 , we have the relation between the Artin and Swan representations (cf. \cite{VS}, p. 248, equation (6.1.9))
\begin{equation}\label{eqn 5.1.22}
 SW_G=A_G+\rm{Ind}_{G_0}^{G}(1)-\rm{Ind}_{\{1\}}^{G}(1),
\end{equation}
$G_0$ is the $0$-th ramification group (i.e., inertia group) of $G$.

Now we are in a position to define the Artin and Swan conductor of a representation $\rho\in R(G)$. The Artin conductor of a 
representation $\rho\in R(G)$ is defined by 
$$a_F(\rho):=<A_G,\rho>_G=<a_G,\chi>_G,$$
where $\chi$ is the character of 
the representation $\rho$. Similarly, for the representation $\rho$, the Swan conductor is:
$$\rm{sw}_F(\rho):=<SW_G,\rho>_G=<\rm{sw}_G,\chi>_G.$$
For more details about Artin and Swan conductor, see Chapter 6 of \cite{VS} and Chapter VI of \cite{JPS}.
\end{dfn}
From equation (\ref{eqn 5.1.22}) we obtain
\begin{equation}\label{eqn 5.1.23}
 a_F(\rho)=\rm{sw}_F(\rho)+\rm{dim}(\rho)-<1,\rho>_{G_0}.
\end{equation}
Moreover, from Corollary of Proposition 4 on p. 101 of \cite{JPS}, for an induced representation 
$\rho:=\rm{Ind}_{\rm{Gal}(K/E)}^{\rm{Gal}(K/F)}(\rho_E)=\rm{Ind}_{E/F}(\rho_E)$, we have
\begin{equation}\label{eqn 5.1.24}
 a_F(\rho)=f_{E/F}\cdot \left( d_{E/F}\cdot \rm{dim}(\rho_E)+\textrm{a}_E(\rho_E)\right).
\end{equation}
We apply this formula (\ref{eqn 5.1.24}) for $\rho_E=\chi_E$ of dimension $1$ and then conversely 
$$a(\chi_E)=\frac{a_F(\rho)}{f_{E/F}}-d_{E/F}.$$
So if we know $a_F(\rho)$ then we can compute $a(\chi_E)$. 


Let $\{G^i\}$, where $i\ge 0,\in\bbQ$ be the ramification subgroups (in the upper numbering) of a local Galois group $G$.
Now let $\rho$ be an irreducible representation of $G$. For this irreducible $\rho$ we define 
$$j(\rho):=\rm{max}\{ i\;|\; \rho|_{G^i}\not\equiv 1\}.$$
Now if $\rho$ is an irreducible representation of $G$, then $\rho|_{I}\not\equiv 1$, where $I=G^0=G_0$ is the inertia subgroup
of $G$. Thus from the definition of $j(\rho)$ we can say, if $\rho$ is irreducible, then we always have 
$j(\rho)\ge 0$, i.e., $\rho$ is nontrivial on the inertia group $G_0$. Then from the definitions of Swan and Artin 
conductors, and equation (\ref{eqn 5.1.23}), when $\rho$ is irreducible, we have the following relations
\begin{equation}\label{eqn 5.1.281}
 \rm{sw}_F(\rho)=\rm{dim}(\rho)\cdot j(\rho),\qquad a_F(\rho)=\rm{dim}(\rho)\cdot (j(\rho)+1).
\end{equation}
From the Theorem of Hasse-Arf (cf. \cite{JPS}, p. 76), if $\rm{dim}(\rho)=1$, i.e., $\rho$ is a character of $G/[G,G]$, 
we can say that $j(\rho)$ must be an integer, then $\rm{sw}_F(\rho)=j(\rho), a_F(\rho)=j(\rho)+1$.
Moreover, by class field theory, $\rho$ corresponds to a linear character $\chi_F$, hence for linear character $\chi_F$, we can write 
$$j(\chi_F):=\rm{max}\{i\;|\;\chi_F|_{U_F^i}\not\equiv1\},$$
because under class field theory (under Artin isomorphism) 
the upper numbering in the filtration of $\rm{Gal}(F_{ab}/F)$ is compatible with the filtration (descending chain) of the group of units 
$U_F$.

From equation (\ref{eqn 5.1.281}),
it is easy to see that for higher dimensional $\rho$, we have $\rm{sw}_F(\rho), a_F(\rho)$ multiples of $\rm{dim}(\rho)$ if and only 
if $j(\rho)$ is an integer.

Now we come to our Heisenberg representations. For each $X\in\widehat{FF^\times}$ we define
\begin{equation}
 j(X):=\begin{cases}
        0 & \text{when $X$ is trivial}\\
        \rm{max}\{i\;|\; X|_{UU^i}\not\equiv 1\} & \text{when $X$ is nontrivial},
       \end{cases}
\end{equation}
where $UU^i\subseteq FF^\times$ is a subgroup which under (\ref{eqn 5.1.2}) corresponds 
$$G_F^i\cap[G_F,G_F]/G_F^i\cap[[G_F,G_F],G_F]\subseteq[G_F,G_F]/[[G_F,G_F],G_F].$$
Let $\rho=\rho(X_\rho,\chi_K)$ be the {\bf minimal conductor} (i.e., a representation with the smallest Artin conductor) 
Heisenberg representation for $X_\rho$ of the absolute Galois group $G_F$. 
From Theorem 3 on p. 125 of \cite{Z5}, we 
have 
\begin{equation}\label{eqn 5.1.26}
 \rm{sw}_F(\rho)=\rm{dim}(\rho)\cdot j(X_\rho)=\sqrt{[F^\times:\rm{Rad}(X_\rho)]}\cdot j(X_\rho).
\end{equation}
Let $\rho_0=\rho_0(X,\chi_0)$ be a minimal representation
corresponding $X$, then all other Heisenberg
representations of dimension $\rm{dim}(\rho)$ are of the form $\rho=\chi_F\otimes \rho_0=(X, (\chi_F\circ N_{K/F})\chi_0)$,
where $\chi_F:F^\times\to \bbC^\times$. Then 
we have (cf. \cite{Z2}, p. 305, equation (5))
\begin{equation}\label{eqn 5.1.27}
 \rm{sw}_F(\rho)=\rm{sw}_F(\chi_F\otimes\rho_0)=\sqrt{[F^\times:\rm{Rad}(X)]}\cdot\rm{max}\{j(\chi_F), j(X)\}.
\end{equation}


For minimal conductor U-isotopic Heisenberg representation we have the following proposition.

\begin{prop}\label{Proposition conductor}
 Let $\rho=\rho(X_\eta,\chi_K)$ be a U-isotropic Heisenberg representation of $G_F$ of minimal conductor. 
 Then we have the following conductor relation
 \begin{center}
  $j(X_\eta)=j(\eta)$, $\rm{sw}_F(\rho)=\rm{dim}(\rho)\cdot j(X_\eta)=\#\eta\cdot j(\eta)$,
  $a_F(\rho)=\rm{sw}_F(\rho)+\rm{dim}(\rho)=\#\eta(j(\eta)+1)=\#\eta\cdot a_F(\eta)$.
 \end{center}
 
\end{prop}

\begin{proof}
 From \cite{Z5}, on p. 126, Proposition 4(i) and Proposition 5(ii), and $U\wedge U=U^1\wedge U^1$, we see the injection 
$U^i\wedge F^\times\subseteq UU^i$ induces a natural isomorphism 
$$U^i\wedge<\pi_F>\cong UU^{i}/UU^i\cap (U\wedge U)$$
for all $i\ge 0$. 

Now let $j(X_\eta)=n-1$, hence $X_\eta|_{UU^n}=1$ but $X_\eta|_{UU^{n-1}}\ne 1$.
This gives $X_\eta|_{U^n\wedge<\pi_F>}=1$ but $X_\eta|_{U^{n-1}\wedge<\pi_F>}\ne 1$. Now from equation (\ref{eqn 5.1.25})
we can conclude that $\eta(x)=1$ for all $x\in U^n$ but $\eta(x)\ne 1$ for $x\in U^{n-1}$. Hence 
$$j(\eta)=n-1=j(X_\eta).$$
Again from the definition of $j(\chi)$, where $\chi$ is a character of $F^\times$, we can see that 
$j(\chi)=a(\chi)-1$, i.e., $a(\chi)=j(\chi)+1$.

From equation (\ref{eqn 5.1.26}) we obtain:
$$\rm{sw}_F(\rho)=\rm{dim}(\rho)\cdot j(X_\eta)=\#\eta\cdot j(\eta),$$
since $\rm{dim}(\rho)=\#\eta$ and $j(X_\eta)=j(\eta)$. Finally, from  equation (\ref{eqn 5.1.23}) for $\rho$ (here $<1,\rho>_{G_0}=0$),
we have 
\begin{equation}\label{eqn 5.1.28}
 a_F(\rho)=\rm{sw}_F(\rho)+\rm{dim}(\rho)=\#\eta\cdot j(\eta)+\#\eta=\#\eta\cdot (j(\eta)+1)=\#\eta\cdot a_F(\eta).
\end{equation}
\end{proof}

By using the equation (\ref{eqn 5.1.24}) in our Heisenberg setting, we have the following proposition.

\begin{prop}\label{Proposition 5.1.20}
 Let $\rho=\rho(Z,\chi_\rho)=\rho(X,\chi_K)$ be a Heisenberg representation of the absolute Galois group $G_F$ of a field 
 $F/\bbQ_p$ of dimension $m$. Let $E/F$ be any subextension in $K/F$ corresponding to a maximal isotropic subgroup for $X$. Then 
 $$a_F(\rho)=a_F(\rm{Ind}_{E/F}(\chi_E)),\qquad m\cdot a_F(\rho)=a_F(\rm{Ind}_{K/F}(\chi_K)).$$
 As a consequence we have 
 $$a(\chi_K)=e_{K/E}\cdot a(\chi_E)-d_{K/E}.$$
\end{prop}
\begin{proof}
 We know that $\rho=\rm{Ind}_{E/F}(\chi_E)$ and $m \cdot \rho=\rm{Ind}_{K/F}(\chi_K)$.
 By the definition of Artin conductor we can write 
 $$a_F(\rm{dim}(\rho)\cdot \rho)=\rm{dim}(\rho)\cdot a_F(\rho)=m\cdot a_F(\rm{Ind}_{E/F}(\chi_E)).$$
 Since $K/E/F$ is a tower of Galois extensions with $[K:E]=m=e_{K/E}f_{K/E}$, we have the transitivity relation of 
 different (cf. \cite{JPS}, p. 51,
 Proposition 8)
 $$\mathcal{D}_{K/F}=\mathcal{D}_{K/E}\cdot \mathcal{D}_{E/F}.$$
 Now from the definition of different of a Galois extension, and taking $K$-valuation we obtain:
 \begin{equation}\label{eqn discriminant relation}
  d_{K/F}=d_{K/E}+e_{K/E}\cdot d_{E/F}.
 \end{equation}
 Now by using equation (\ref{eqn 5.1.24}) we have:
 \begin{equation}\label{eqn 44}
  m\cdot a_F(\rm{Ind}_{E/F}(\chi_E))=m\cdot f_{E/F}\left(d_{E/F}+a(\chi_E)\right)=m\cdot f_{E/F}\cdot d_{E/F}+e_{K/E}\cdot f_{K/F}
  \cdot a(\chi_E),
 \end{equation}
and 
\begin{equation}\label{eqn 45}
 a_F(\rm{Ind}_{K/F}(\chi_K))=f_{K/F}\cdot\left(d_{K/F}+a(\chi_K)\right)=f_{K/F}\cdot d_{K/F}+f_{K/F}\cdot a(\chi_K).
\end{equation}
By using equation (\ref{eqn discriminant relation}), from equations (\ref{eqn 44}), (\ref{eqn 45}), we have 
$$a(\chi_K)=e_{K/E}\cdot a(\chi_E)-d_{K/E}$$

\end{proof}

Now by combining Proposition \ref{Proposition 5.1.20} with Proposition \ref{Proposition conductor}, we get the following result.

\begin{lem}\label{Lemma general conductor}
 Let $\rho=\rho(X_\eta,\chi_K)$ be a U-isotopic Heisenberg representation of the absolute Galois group $G_F$ of a non-archimedean
 local field $F$.
 Let $K=K_\eta$ correspond to the radical of $X_\eta$, and let $E_1/F$ be the maximal unramified subextension, and $E/F$
 be any maximal cyclic and totally ramified subextension in $K/F$. Let $m$ denote the order of $\eta$.
 Then $\rho$ is induced by $\chi_{E_1}$ or by 
 $\chi_E$ respectively, and we have 
 \begin{enumerate}
  \item $a_E(\chi_E)=m\cdot a(\eta)-d_{E/F}$,
  \item $a_{E_1}(\chi_{E_1})=a(\eta)$,
  \item and for the character $\chi_K\in\widehat{K^\times}$,
  $$a_K(\chi_K)=m\cdot a(\eta)-d_{K/F}.$$
 \end{enumerate}
Moreover, $a_E(\chi_E)=a_K(\chi_K)$. 
\end{lem}
\begin{proof}
Proof of these assertions follows from equation (\ref{eqn 5.1.24}) and Proposition \ref{Proposition conductor}. When 
$\rho=\rm{Ind}_{E/F}(\chi_E)$, where $E/F$ is a maximal cyclic and totally ramified subextension in $K/F$, from equation 
(\ref{eqn 5.1.24}) we have
\begin{align*}
 a_F(\rho)
 &=m\cdot a(\eta)\quad\text{using Proposition $\ref{Proposition conductor}$},\\
 &=f_{E/F}\cdot\left(d_{E/F}\cdot 1+a_E(\chi_E)\right),\quad\text{since $\rho=\rm{Ind}_{E/F}(\chi_E)$}\\
 &=1\cdot\left(d_{E/F}+a_E(\chi_E)\right).
\end{align*}
because $E/F$ is totally ramified, hence $f_{E/F}=1$.  This implies $a_E(\chi_E)=m\cdot a(\eta)-d_{E/F}$.

Similarly, when $\rho=\rm{Ind}_{E_1/F}(\chi_{E_1})$, where $E_1/F$ is the maximal unramified subextension in $K/F$, hence 
$f_{E_1/F}=m$ and $d_{E_1/F}=0$, by using equation (\ref{eqn 5.1.24}) we obtain $a_{E_1}(\chi_{E_1})= a(\eta)$.

Again from Proposition \ref{Proposition 5.1.20} we have 
$$a_K(\chi_K)=m\cdot a(\chi_{E_1})-d_{K/E_1}=m\cdot a(\eta)-d_{K/F}.$$

Finally, since $E/F$ is a maximal cyclic totally ramified implies $K/E$ is unramified and therefore 
$$d_{E/F}=d_{K/F},\quad\text{and hence}\; a_E(\chi_E)=a_K(\chi_K).$$
\end{proof}

\begin{rem}\label{Remark 5.1.22}
 Assume that we are in the dimension $m=\#\eta$ prime to $p$ case. Then from Corollary \ref{Corollary U-isotropic}, $\eta$
must be a character of $U/U^1$ (for $U=U_F$), hence
$$ a(\eta)=1\qquad  a_F(\rho_0) =m.$$
Therefore in this case the minimal conductor of $\rho$ is $m$, hence it is equal to the dimension of $\rho$. 

From the above Lemma \ref{Lemma general conductor}, in this case we have 
$$a_{E_1}(\chi_{E_1})=a(\eta)=1.$$
And $K/F, E/F$ are tamely ramified of ramification exponent $e_{K/F}=m$, hence
$$ a_E(\chi_E) = a_K(\chi_K) = m\cdot a(\eta)-d_{K/F}=m -(e_{K/F}-1)=m-(m-1)=1.$$
Thus we can conclude that in this case all three characters (i.e., $\chi_{E_1},\chi_E$, and $\chi_K$) are of conductor $1$.

In the general case $a_{E_1}(\chi_{E_1}) = a(\eta)$ and
$$a_E(\chi_E)= a_K(\chi_K) = m\cdot a(\eta)-d,$$
where $d=d_{E/F}=d_{K/F}$, conductors will be different.
\end{rem}

In general, if $\rho=\rho_0\otimes\chi_F$, where $\rho_0$ is a finite dimensional minimal conductor representation of $G_F$, and 
$\chi_F\in\widehat{F^\times}$, then we have the following result.

\begin{lem}\label{Lemma 5.1.23}
 Let $\rho_0$ be a finite dimensional representation of $G_F$ of minimal conductor.
 Then we have 
 \begin{equation}
  a_F(\rho)=\rm{dim}(\rho_0)\cdot a_F(\chi_F),
 \end{equation}
where $\rho=\rho_0\otimes\chi_F=\rho(X_\eta,(\chi_F\circ N_{K/F})\chi_0)$ and $\chi_F\in\widehat{F^\times}$ with 
$a(\chi_F)>\frac{a(\rho_0)}{\rm{dim}(\rho)}$.
\end{lem}
\begin{proof}
From equation (\ref{eqn 5.1.281}) we have $a_F(\rho_0)=\rm{dim}(\rho_0)\cdot (1+j(\rho_0))$.
By the given condition $\rho_0$ is of minimal conductor. So for representation $\rho=\rho_0\otimes\chi_F$, we have 
\begin{align*}
 a_F(\rho)
 &=a_F(\rho_0\otimes\chi_F)=\rm{dim}(\rho_0)\cdot\left(1+\rm{max}\{j(\rho_0),j(\chi_F)\}\right)\\
 &=\rm{dim}(\rho_0)\cdot\rm{max}\{1+j(\chi_F), 1+j(\rho_0)\}\\
 &=\rm{dim}(\rho_0)\cdot\rm{max}\{a(\chi_F), 1+j(\rho_0)\}\\
 &=\rm{dim}(\rho_0)\cdot a_F(\chi_F),
\end{align*}
because by the given condition 
$$a(\chi_F)>\frac{a(\rho_0)}{\rm{dim}(\rho_0)}=\frac{\rm{dim}(\rho_0)\cdot(1+j(\rho_0))}{\rm{dim}(\rho_0)}=1+j(\rho_0).$$

\end{proof}

\begin{prop}\label{Proposition 5.1.23}
 Let $\rho=\rho(X,\chi_K)$ be a Heisenberg representation dimension $m$ of the absolute Galois group $G_F$ of a 
 non-archimedean local field $F$.
 Then $m| a_F(\rho)$ if and only if:\\
$X$ is $U$-isotropic, or (if $X$ is not $U$-isotropic) $a_F(\rho)$ is with respect to $X$ not the minimal conductor.
\end{prop}
\begin{proof}
From the above Lemma \ref{Lemma 5.1.23} we know that if $\rho$ is not minimal, then $a_F(\rho)$ is always a multiple of the 
dimension $m$. So now we just have to check for minimal conductors. In the U-isotropic case the minimal conductor is multiple
of the dimension (cf. Proposition \ref{Proposition conductor}). 

Finally, suppose that $X$ is not U-isotropic, i.e., $X|_{U\wedge U}=X|_{U^1\wedge U^1}\not\equiv1$, because 
$U\wedge U=U^1\wedge U^1$ (see the Remark on p. 126 of \cite{Z5}). We also know that 
$UU^i=(UU^i\cap U^1\wedge U^1)\times(U^i\wedge<\pi_F>)$ (cf. \cite{Z5}, p. 126, Proposition 5(ii)). 
In Proposition 5 of \cite{Z5}, we observe that all the jumps $v$ in the filtration $\{UU^i\cap (U^1\wedge U^1)\}, i\in\bbR_{+}$
are not {\bf integers with $v>1$}. This shows that $j(X)$ is also not an integer, hence $a_F(\rho_0)$ is not 
multiple of the dimension. This implies the conductor $a_F(\rho)$ is not minimal.

\end{proof}

Let $\rho=\rho(X,\chi_K)$ be a Heisenberg representation of the absolute Galois group $G_F$. Then from equation 
(\ref{eqn dimension formula}), we have
$$\rm{dim}(\rho)=\sqrt{[K:F]}=\sqrt{[F^\times:\cN_{K/F}]},$$
when $\cN_{K/F}=\rm{Rad}(X)$.

\begin{lem}\label{Lemma dimension equivalent}
 Let $\rho=(Z_\rho,\chi)=\rho(X_\rho,\chi)$ be a Heisenberg representation of 
 the absolute Galois group $G_F$ of a non-archimedean local field 
 $F/\bbQ_p$. Then following are equivalent:
 \begin{enumerate}
  \item $\rm{dim}(\rho)$ is prime to $p$.
  \item $\rm{dim}(\rho)$ is a divisor of $q_F-1$.
  \item The alternating character $X_\rho$ is $U$-isotropic and $X_\rho=X_\eta$ for a character $\eta$ of 
  $U_F/U_F^1$.
 \end{enumerate}
\end{lem}
\begin{proof}
 From Corollary \ref{Corollary U-isotropic} we know that all Heisenberg representations of dimensions prime to $p$, are 
 U-isotropic representations of the form $\rho=\rho(X_\eta,\chi)$, where $\eta:U_F/U_F^1\to\bbC^\times$, and the dimensions 
 $\rm{dim}(\rho)=\#\eta$.
 
 Thus if $\rm{dim}(\rho)$ is prime to $p$, then $\rm{dim}(\rho)=\#\eta$ is a divisor of $q_F-1$. And if $\rm{dim}(\rho)$
 is a divisor of $q_F-1$, then $gcd(p,\rm{dim}(\rho))=1$. Then from Corollary \ref{Corollary U-isotropic}, the alternating 
 character $X_\rho$ is U-isotropic and $X_\rho=X_\eta$ for a character $\eta\in\widehat{U_F/U_F^1}$.
 
 Finally, if $\rho=\rho(X_\rho,\chi_K)=\rho(X_\rho,\chi_K)$ be a Heisenberg representation of $G_F$
for a character $\eta$ of $U_F/U_F^1$, then from Corollary \ref{Corollary U-isotropic}, we know that 
$\rm{dim}(\rho)$ is prime to $p$.
\end{proof}

For giving invariant formula of $W(\rho)$, we need to know the explicit dimension formula of $\rho$.
In the following theorem we give the general dimension formula of a Heisenberg representation.
\begin{thm}[{\bf Dimension}]\label{Dimension Theorem}
Let $F/\bbQ_p$ be a local field and $G_F$ be the absolute Galois group of $F$. If $\rho$ is a Heisenberg representation of 
$G_F$, then $\rm{dim}(\rho)=p^n\cdot d'$, where $n\ge 0$ is an integer and where the prime to $p$ factor $d'$ must divide $q_F-1$.
\end{thm}
\begin{proof}
 By the definition of Heisenberg representation $\rho$ we have the relation 
 $$[[G_F,G_F],G_F]\subseteq\rm{Ker}(\rho).$$
 Then we can consider $\rho$ as a representation of $G:=G_F/[[G_F,G_F],G_F]$. Since 
 $[x,g]\in [[G_F,G_F],G_F]$ for all $x\in [G_F,G_F]$ and $g\in G_F$, we have $[G,G]=[G_F,G_F]/[[G_F,G_F],G_F]\subseteq Z(G)$,
 hence $G$ is a two-step nilpotent group.
 
 We know that each nilpotent group is isomorphic to the direct product of its Sylow subgroups. Therefore we can write 
 $$G=G_p\times G_{p'},$$
 where $G_p$ is the Sylow $p$-subgroup, and $G_{p'}$ is the direct product of all other Sylow subgroups. Therefore each irreducible
 representation $\rho$ has the form $\rho=\rho_{p}\otimes\rho_{p'}$, where $\rho_{p}$ and $\rho_{p'}$ are irreducible representations of 
 $G_p$ and $G_{p'}$ respectively. 
 
 We also know that finite $p$-groups are nilpotent groups, and direct product of a finite number of 
 nilpotent groups is again a nilpotent group.
 So $G_p$ and $G_{p'}$ are both two-step nilpotent group because $G$ is a two-step nilpotent group. Therefore the representations
 $\rho_p$ and $\rho_{p'}$ are both Heisenberg representations of $G_p$ and $G_{p'}$ respectively.
 
 Now to prove our assertion, we have to show that $\rm{dim}(\rho_p)$ can be an arbitrary power of $p$, whereas 
 $\rm{dim}(\rho_{p'})$ must divide $q_F-1$. Since
 $\rho_p$ is an {\bf irreducible} representation of $p$-group $G_p$, so the dimension of $\rho_p$ is some $p$-power.
 
 Again from the construction of $\rho_{p'}$ we can say that $\rm{dim}(\rho_{p'})$ is {\bf prime} to $p$. 
 Then from Lemma \ref{Lemma dimension equivalent} $\rm{dim}(\rho_{p'})$ is a divisor of $q_F-1$.

This completes the proof.

\end{proof}

\begin{rem}\label{Remark 5.1.3}
{\bf (1).}
Let $V_F$ be the wild ramification subgroup of $G_F$.
 We can show that $\rho|_{V_F}$ is irreducible if and only if $Z_\rho=G_K\subset G_F$
 corresponds to an abelian extension $K/F$ which is totally ramified and wildly 
 ramified\footnote{Group theoretically, if $\rho|_{V_F}=\rm{Ind}_{H}^{G_F}(\chi_H)|_{V_F}$ is irreducible, then from 
 Section 7.4 of \cite{JP},
we can say $G_F=H\cdot V_F$. Here $H=G_L$, where $L$ is a certain extension of $F$, and $V_F=G_{F_{mt}}$ where $F_{mt}/F$ is the maximal 
tame extension of $F$. Therefore $G_F=H\cdot V_F$ is equivalent to $F=L\cap F_{mt}$ that means the extension $L/F$ must be totally 
ramified and wildly ramified, and $[G_F:H]=[L:F]=|V_F|$.
We know that the wild ramification subgroup $V_F$ is a pro-p-group (cf. \cite{FV}, p. 106). Then 
 $\rm{dim}(\rho)$ is a power of $p$.} (cf. \cite{Z2}, p. 305). If $N:=N_{K/F}(K^\times)$ is the subgroup
 of norms, then this means that $N\cdot U_{F}^{1}=F^\times$, in other words,
 $$F^\times/N=N\cdot U_{F}^{1}/N=U_{F}^{1}/N\cap U_{F}^{1},$$
 where $N$ can be also considered as the radical of $X_\rho$. So we can consider the alternating character $X_\rho$ on the principal
 units $U_{F}^{1}\subset F^\times$. Then 
 $$\rm{dim}(\rho)=\sqrt{[F^\times:N]}=\sqrt{[U_F^1: N\cap U_F^1]},$$
 is a power of $p$, because $U_F^1$ is a pro-p-group.

 Here we observe: If $\rho=\rho(X,\chi_K)$ with $\rho|_{V_F}$ stays irreducible, then
 $\rm{dim}(\rho)=p^n$, $n\ge 1$ and  
 $K/F$ is a totally and {\bf wildly} ramified. But there is 
 a {\bf big} class of Heisenberg representations $\rho$ such that $\rm{dim}(\rho)=p^n$ is a $p$-power, but which are not 
 wild representations (see the Definition \ref{Definition U-isotropic} of U-isotropic).\\
 {\bf (2).}
Let $\rho=\rho(X,\chi_K)$ be a Heisenberg representation of the absolute Galois group $G_F$ of dimension $d>1$ which is prime 
to $p$. Then from above Lemma \ref{Lemma dimension equivalent}, we have  $d|q_F-1$. For this representation $\rho$, here 
$K/F$ must be tame if $\rm{Rad}(X)=\cN_{K/F}$ (cf. \cite{FV}, p. 115).
\end{rem}

\section{\textbf{Invariant formula for $W(\rho)$}}

\begin{lem}\label{Lemma 5.2.1}
Let $\rho=\rho(Z,\chi_Z)$ be a Heisenberg representation of the local Galois group $G=\mathrm{Gal}(L/F)$ of odd dimension.
Let $H$ be a maximal isotropic subgroup for $\rho$ and $\chi_H\in\widehat{H}$ with $\chi_H|_{Z}=\chi_Z$
then:
 \begin{equation}\label{eqn 4.9}
  W(\rho)=W(\chi_H),\hspace{.5cm} W(\rho)^{\mathrm{dim}(\rho)}=W(\chi_Z),
 \end{equation}
 and 
 \begin{equation}
  W(\chi_H)^{[H:Z]}=W(\chi_Z).
 \end{equation}
 \end{lem}
\begin{proof}
From the construction of Heisenberg representation $\rho$ of $G$ we have 
\begin{center}
 $\rho=\rm{Ind}_{H}^{G}(\chi_H)$, \hspace{.4cm} $\rm{dim}(\rho)\cdot\rho=\rm{Ind}_{Z}^{G}(\chi_Z)$.\\
 This implies that $W(\rho)=\lambda_{H}^{G}\cdot W(\chi_H)$ and $W(\rho)^{\rm{dim}(\rho)}=\lambda_{Z}^{G}\cdot W(\chi_Z).$
\end{center}
Since $\rm{dim}(\rho)$ is odd we may apply now Lemma 3.4 on p. 10 of \cite{SAB1}, and we obtain 
 \begin{center}
  $\lambda_{H}^{G}=\lambda_{Z}^{G}=1$.
 \end{center}
 So, we have $W(\rho)=\lambda_{H}^{G}(W)\cdot W(\chi_H)=W(\chi_H)$. Similarly, we have 
 $W(\rho)^{\mathrm{dim}(\rho)}=W(\chi_Z)$.
 
Moreover, it is easy to see\footnote{We have 
 \begin{center}
  $d\cdot\rho=\mathrm{Ind}_{Z}^{G}\chi_Z=\mathrm{Ind}_{H}^{G}(\mathrm{Ind}_{Z}^{H}\chi_Z)$,
 \end{center}
 and $\mathrm{Ind}_{Z}^{H}\chi_Z$ of dimension $d=[H:Z]$. Therefore:
 \begin{center}
  $W(\rho)^d=(\lambda_{H}^{G})^d\cdot W(\mathrm{Ind}_{Z}^{H}\chi_Z)$.
 \end{center}
On the other hand $W(\rho)=\lambda_{H}^{G}\cdot W(\chi_H)$ implies
\begin{center}
 $W(\rho)^d=(\lambda_{H}^{G})^d\cdot W(\chi_H)^d$.
\end{center}
Now comparing these two expressions for $W(\rho)^d$ we see that 
\begin{center}
 $W(\chi_H)^d=W(\mathrm{Ind}_{Z}^{H}\chi_Z)$.
\end{center}} that $W(\rm{Ind}_{Z}^{H}(\chi_Z))=W(\chi_H)^{[H:Z]}$.
By the given condition, $[H:Z]=\rm{dim}(\rho)$ is odd, hence $\lambda_{Z}^{H}=1$, 
then we have 
\begin{equation}\label{eqn 5.2.4}
 W(\chi_H)^{[H:Z]}=W(\rm{Ind}_{Z}^{H}(\chi_Z))=W(\chi_Z).
\end{equation}

 \end{proof}
 
\begin{rem}
 Related to $G\supset H\supset Z$ we have the base fields $F\subset E\subset K$, and $\chi_Z$ is the restriction of 
 $\chi_H$. In arithmetic terms this means:
 $$\chi_K=\chi_E\circ N_{K/E}.$$
 So in arithmetic terms of $W(\rm{Ind}_{Z}^{G}(\chi_Z))=W(\rm{Ind}_{H}^{G}(\chi_H))^{[G:H]}$ is as follows:
$$W(\rm{Ind}_{K/F}(\chi_K),\psi)=W(\rm{Ind}_{E/F}(\chi_E),\psi)^{[K:E]}.$$
Then we can conclude that 
$$\lambda_{K/E}\cdot W(\chi_K,\psi_K)=W(\chi_E,\psi_E)^{[K:F]}.$$
If the dimension $\rm{dim}(\rho)=[K:E]$ is odd, we have $\lambda_{K/E}=1$, because $K/E$ is Galois.
Then we obtain
\begin{equation}\label{eqn 5.2.5}
  W(\chi_E,\psi_E)^{[K:E]}=W(\chi_E\circ N_{K/E},\psi_E\circ\rm{Tr}_{K/E}).
 \end{equation}
The formula (\ref{eqn 5.2.5}) is known as a {\bf Davenport-Hasse} relation (cf. \cite{LN}, p. 197, Theorem 5.14).

\end{rem}

\begin{cor}
 Let $\rho=\rho(Z,\chi_Z)$ be a Heisenberg representation of a local Galois group $G$. Let $\rm{dim}(\rho)=d$ be odd.
 Let the order of $W(\chi_Z)$ be $n$ (i.e., $W(\chi_Z)^n=1$). If $d$ is prime to $n$, then $d^{\varphi(n)}\equiv 1\mod{n}$, and 
 $$W(\rho)=W(\chi_Z)^{\frac{1}{d}}=W(\chi_Z)^{d^{\varphi(n)-1}},$$
where $\varphi(n)$ is the Euler's totient function of $n$.
\end{cor}
 \begin{proof}
 By our assumption, here $d$ and $n$ are coprime. Therefore from {\bf Euler's theorem} we can write 
 $$d^{\varphi(n)}\equiv 1\mod{n}.$$
 This implies $d^{\varphi(n)}-1$ is a multiple of $n$.
 
 Here $d$ is odd, then from the above Lemma \ref{Lemma 5.2.1} we have $W(\rho)^d=W(\chi_Z)$.
 So we obtain
 $$W(\rho)=W(\chi_Z)^{\frac{1}{d}}=W(\chi_Z)^{d^{\varphi(n)-1}},$$
 since $d^{\varphi(n)}-1$ is a multiple of $n$, and by assumption $W(\chi_Z)^n=1$.
\end{proof}

We observe that when $\rm{dim}(\rho)=d$ is odd, if we take second part of the 
equation (\ref{eqn 4.9}), we have $W(\rho)=W(\chi_Z)^{\frac{1}{d}}$, but it is not well-defined in general. 
Here we have to make precise 
which root $W(\chi_Z)^{\frac{1}{d}}$ really occurs. That is why, giving invariant formula of
$W(\rho)$ using $\lambda$-functions computation is difficult. In the following theorem we give an invariant formula of local 
constant for Heisenberg representation.

\begin{thm}\label{Theorem invariant odd}
 Let $\rho=\rho(X,\chi_K)$ be a Heisenberg representation of the absolute Galois group $G_F$ of a local field $F/\bbQ_p$
 of dimension $d$. Let $\psi_F$ be the canonical additive character of $F$ and $\psi_K:=\psi_F\circ\rm{Tr}_{K/F}$.
 Denote $\mu_{p^\infty}$ as the group of roots of unity of $p$-power order and $\mu_{n}$ as the group of 
 $n$-th roots of unity, where $n>1$ is an integer.
 \begin{enumerate}
  \item When the dimension $d$ is odd, we have 
   \begin{center}
  $W(\rho)\equiv W(\chi_\rho)'\mod{\mu_{d}}$,
 \end{center}
where $W(\chi_\rho)'$ is any $d$-th root of 
$W(\chi_K,\psi_K)$.
\item When the dimension $d$ is even, we have 
 \begin{center}
  $W(\rho)\equiv W(\chi_\rho)'\mod{\mu_{d'}}$,
 \end{center}
 where $d'=\rm{lcm}(4,d)$.
 \end{enumerate}
 
\end{thm}

\begin{proof}
{\bf (1).}
We know that the lambda functions are always fourth roots of unity. In particular, when degree of the Galois extension 
$K/F$ is odd, from Theorem \ref{General Theorem for odd case} we have $\lambda_{K/F}=1$. For proving our assertions we will use these 
facts about $\lambda$-functions.

We know that $\rm{dim}(\rho)\cdot\rho=\rm{Ind}_{K/F}(\chi_K)$, where by class field theory $\chi_K\leftrightarrow\chi_\rho$ 
 is a character of $K^\times$.
When $d$ is odd, we can write 
 $$W(\rho)^d=\lambda_{K/F}\cdot W(\chi_K,\psi_K)= W(\chi_K,\psi_K).$$
 Now let $W(\chi_\rho)'$ be any $d$-th root of $W(\chi_K,\psi_K)$. Then we have 
 $$W(\rho)^d={W(\chi_\rho)'}^d,$$
 hence $\frac{W(\rho)}{W(\chi_\rho)}$ is a $d$-th root of unity. Therefore we have
 $$W(\rho)\equiv W(\chi_\rho)' \mod{\mu_{d}}.$$
 {\bf (2).}
Similarly, we can give invariant formula for even degree Heisenberg representations. When the dimension $d$ of $\rho$ is even, we have 
\begin{equation}\label{eqn 5.2.10}
 W(\rho)^d=\lambda_{K/F}\cdot W(\chi_K,\psi_K)\equiv W(\chi_K,\psi_K)\mod{\mu_4},
\end{equation}
because $\lambda_{K/F}$ is a fourth root of unity.
Now let $W(\chi_\rho)'$ be any $d$-th root of $W(\chi_K,\psi_K)$, hence $W(\chi_K,\psi_K)=W(\chi_\rho)'^d$. Then from equation 
(\ref{eqn 5.2.10}) we have 
$$\left(\frac{W(\rho)}{W(\chi_\rho)'}\right)^d\equiv 1\mod{\mu_4}.$$
Therefore we can conclude that 
\begin{equation}
 W(\rho)\equiv W(\chi_\rho)'\mod{\mu_{d'}},
\end{equation}
where $d'=\rm{lcm}(4, d)$.\\

\end{proof}

When dimension of a Heisenberg representation $\rho=\rho(X,\chi_K)$ of $G_F$ is prime to $p$, then from Lemma 
\ref{Lemma dimension equivalent} we can say that $X=X_\eta$ is U-isotropic with $\eta:U_F/U_F^1\to\bbC^\times$. Again from 
Remark \ref{Remark 5.1.22} we observe that $a(\chi_K)=1$ when $\rho$ is of minimal conductor.  In the following lemma for minimal 
conductor $\rho$ with dimension prime to $p$, we show that 
$W(\rho)$ is a root of unity.

\begin{lem}\label{Lemma 5.2.12}
 Let $\rho=\rho(X,\chi_K)$ be a minimal conductor Heisenberg representation with respect to $X$ 
 of the absolute Galois group $G_F$ of a non-archimedean local field $F/\bbQ_p$.
 If dimension $\rm{dim}(\rho)$ is prime to $p$, then $W(\rho)$ is always a root of unity.\\
\end{lem}
\begin{proof}
 Assume that $\rm{dim}(\rho)=d$ and $gcd(g,p)=1$. Then from Lemma \ref{Lemma dimension equivalent}, we can say that 
 $\rho=\rho(X,\chi_K)=\rho(X_\eta,\chi_K)$ is a U-isotropic with $a(\eta)=1$. Since $\rho$ is of minimal conductor,
 from Remark \ref{Remark 5.1.22} we have 
 $a(\chi_K)=1$. 
 
 From equation (\ref{eqn 5.1.5}) we also know that: 
\begin{center}
 $d\cdot\rho=\rm{Ind}_{K/F}(\chi_K)$.
\end{center}
Then we can write 
\begin{align}
 W(\rho)^d
 &=\lambda_{K/F}\cdot W(\chi_K)\nonumber\\
 &=\lambda_{K/F}\cdot q_{K}^{-\frac{1}{2}}\sum_{x\in U_K/U_{K}^{1}}\chi_{K}^{-1}(x/c)\psi_K(x/c)\nonumber\\
 &=\lambda_{K/F}\cdot q_{K}^{-\frac{1}{2}}\tau(\chi_K),\label{eqn 4.25}
\end{align}
where $c=\pi_{K}^{1+n(\psi_K)}$, $\psi_K=\psi_F\circ\rm{Tr}_{K/F}$, the canonical character of $K$
and 
\begin{equation}
 \tau(\chi_K)=\sum_{x\in U_K/U_{K}^{1}}\chi_{K}^{-1}(x/c)\psi_K(x/c).
\end{equation}
Since $U_K/U_{K}^{1}\cong k_{K}^{\times}$, $a(\chi_K)=1$, and $n(\frac{1}{c}\cdot\psi_K)=-1$,
we can consider $\tau(\chi_K)$ as a classical Gauss sum 
of $\chi_K$. We also know that $|\tau(\chi_K)|=q_{K}^{\frac{1}{2}}$ (cf. \cite{M}, p. 30, Proposition 2.2(i)).

Moreover, here we have $f_{K/F}=e_{K/F}=d$, hence $f_{K/\bbQ_p}\ge d$. So here we have $q_K=p^{f_{K/\bbQ_p}}\ge p^d$. Then from 
Theorem 1.6.2 on p. 33 of \cite{BRK}, we can write 
\begin{center}
 $\tau(\chi_K)=q_{K}^{\frac{1}{2}}\cdot\gamma$,
\end{center}
where $\gamma$ is a certain root of unity.

We also know that $\lambda_{K/F}^{4}=1$, then from the equation (\ref{eqn 4.25}) we obtain:
\begin{equation}
 W(\rho)^{4 d n}=\gamma^{4n}=1,
\end{equation}
where $n$ is the order of $\gamma$.

 This completes the proof.
\end{proof}

\begin{rem}
 As to the computation of $W(\rho)=W(\rho(X,\chi_K))$ we also can precisely say what an unramified
twist will do by the formula of local constant of unramified character twist (cf. \cite{JT2}, p. 15, (3.4.5)). Let 
$\omega_{K,s}$ be an unramified character of $K^\times$ such that $\omega_{K,s}|_{F^\times}=\omega_{F,s}$, then we have
\begin{equation}
 \omega_{F,s}\otimes\rho(X,\chi_K)=\rho(X,\omega_{K,s}\cdot\chi_K),\hspace{.3cm} W(\rho(X,\omega_{K,s}\cdot\chi_K))=
 \omega_{F,s}(c_{\rho,\psi})\cdot W(\rho(X,\chi_K)).
\end{equation}
Therefore the question: Is $W(\rho)$ a root of unity or not?, is completely under control if
we do unramified twists. In particular, unramified twists of finite order cannot change the
answer.
\end{rem}

In the following theorem we give an invariant formula for $W(\rho,\psi)$, where 
$\rho=\rho(X,\chi_K)$ is a minimal conductor Heisenberg representation of the absolute Galois group $G_F$
of a local field $F/\bbQ_p$ of dimension $m$ which is prime to $p$.

 \begin{thm}\label{invariant formula for minimal conductor representation}
  Let $\rho=\rho(X,\chi_K)$ be a minimal conductor Heisenberg representation of the absolute Galois group $G_F$ of a non-archimedean
  local field $F/\bbQ_p$ of dimension $m$ with $gcd(m,p)=1$. Let $\psi$ be a nontrivial additive character of $F$. Then 
  \begin{equation}
   W(\rho,\psi)=R(\psi,c)\cdot L(\psi,c),
  \end{equation}
where 
$$R(\psi,c):=\lambda_{E/F}(\psi)\Delta_{E/F}(c),$$
is a fourth root of unity that depends on $c\in F^\times$ with $\nu_F(c)=1+n(\psi)$ but not on the totally ramified cyclic subextension
$E/F$ in $K/F$, and 
$$L(\psi,c):=\det(\rho)(c)q_F^{-\frac{1}{2}}\sum_{x\in k_F^\times}(\chi_K\circ N_{E_1/F}^{-1})^{-1}(x)\cdot (c^{-1}\psi)(mx),$$
where $E_1/F$ is the unramified extension of $F$ of degree $m$.
 \end{thm}

 Before proving this Theorem \ref{invariant formula for minimal conductor representation} we need to prove the following
 lemma.
 \begin{lem}\label{Lemma 5.2.17} (With the notation of the above theorem)
  \begin{enumerate}
   \item Let $E/F$ be any totally ramified cyclic extension of degree $m$ inside $K/F$. Then:
   $$\Delta_{E/F}(\epsilon)=:\Delta(\epsilon),\qquad\text{for all $\epsilon\in U_F$},$$
does not depend on $E$ if we restrict to units of $F$.
\item We have $L(\psi,\epsilon c)=\Delta(\epsilon)L(\psi,c)$, and therefore changing $c$ by unit we see that 
$$\Delta_{E/F}(\epsilon c)\cdot L(\psi,\epsilon c)=\Delta(\epsilon)^2\Delta_{E/F}(c)\cdot L(\psi,c)=\Delta_{E/F}(c) L(\psi,c).$$
\item We also have the transformation rule $R(\psi,\epsilon c)=\Delta(\epsilon)\cdot R(\psi,c)$.   
   \end{enumerate}
\end{lem}

 \begin{proof}
  {\bf (1).} Denote $G:=\rm{Gal}(E/F)$. By class field theory we know that 
  \begin{equation}
   \Delta_{E/F}=\begin{cases}
                 \omega_{E'/F} & \text{when $\rm{rk}_2(G)=1$}\\
                 1 & \text{when $\rm{rk}_2(G)=0$},
                \end{cases}
  \end{equation}
where $E'/F$ is a uniquely determined quadratic extension inside $E/F$, and $\omega_{E'/F}$ is the quadratic character of $F^\times$
which corresponds to the extension $E'/F$ by class field theory.

When $m$ is odd, i.e., $\rm{rk}_2(G)=0$, hence $\Delta_{E/F}\cong 1$.
So for odd case, the assertion (1) is obvious. 

When $m$ is even, we choose two different totally ramified cyclic subextensions, namely $L_1/F, \;L_2/F$, in $K/F$ of degree $m$.
Then we can write for all $\epsilon\in U_F$,
$$\Delta_{L_1/F}(\epsilon)=\omega_{E'/F}(\epsilon)
=\eta(\epsilon)\cdot\omega_{E'/F}(\epsilon)=\omega_{E'/F}(\epsilon)=\Delta_{L_2/F}(\epsilon),$$
where $\eta$ is the unramified quadratic character of $F^\times$. This proves that $\Delta_{E/F}$ does not depend on $E$ if we restrict 
to $U_F$.\\
{\bf (2).} From Proposition \ref{Proposition arithmetic form of determinant} we know that 
$\det(\rho)(x)=\Delta_{E/F}(x)\cdot \chi_K\circ N_{K/E}^{-1}(x)$ for all $x\in F^\times$.
Let $E_1/F$ be the unramified subextension in $K/F$ of degree $m$. Then we have $EE_1=K$ and 
$$N_{K/E}|_{E_1}=N_{E_1/F}, \qquad (E_1^\times)_{F}\subseteq K_E^\times\subset\rm{Ker}(\chi_K).$$
Moreover $U_F\subset\cN_{E_1/F}$ and therefore we may write $N_{K/E}^{-1}(\epsilon)=N_{E_1/F}^{-1}(\epsilon)$
for all $\epsilon\in U_F$. Now we can write:
\begin{align*}
 L(\psi,\epsilon c)
 &=\det(\rho)(\epsilon c)q_F^{-\frac{1}{2}}\sum_{x\in k_F^\times}(\chi_K\circ N_{E_1/F}^{-1})^{-1}(x)\cdot (c^{-1}\psi)(mx/\epsilon)\\
 &=\Delta_{E/F}(\epsilon)\chi_K\circ N_{K/E}^{-1}(\epsilon)\det(\rho)(c)q_F^{-\frac{1}{2}}
 \sum_{x\in k_F^\times}(\chi_K\circ N_{E_1/F}^{-1})^{-1}(\epsilon x)\cdot (c^{-1}\psi)(mx)\\
 &=\Delta(\epsilon)\chi_K\circ N_{E_1/F}^{-1}(\epsilon\epsilon^{-1})\det(\rho)(c)q_F^{-\frac{1}{2}}
 \sum_{x\in k_F^\times}(\chi_K\circ N_{E_1/F}^{-1})^{-1}(x)\cdot (c^{-1}\psi)(mx)\\
 &=\Delta(\epsilon)\cdot \det(\rho)(c)q_F^{-\frac{1}{2}}
 \sum_{x\in k_F^\times}(\chi_K\circ N_{E_1/F}^{-1})^{-1}(x)\cdot (c^{-1}\psi)(mx)\\
 &=\Delta(\epsilon)\cdot L(\psi,c).
\end{align*}
This implies that 
$$\Delta_{E/F}(\epsilon c)\cdot L(\psi,\epsilon c)=\Delta(\epsilon)^2\Delta_{E/F}(c)\cdot L(\psi,c)=\Delta_{E/F}(c)L(\psi,c).$$
{\bf (3).} By the definition of $R(\psi,c)$ we can write:
\begin{align*}
 R(\psi,\epsilon c)
 &=\lambda_{E/F}(\psi)\Delta_{E/F}(\epsilon c)=\lambda_{E/F}(\psi)\Delta_{E/F}(\epsilon)\Delta_{E/F}(c)\\
 &=\Delta(\epsilon)\lambda_{E/F}(\psi)\Delta_{E/F}(c)=\Delta(\epsilon)\cdot R(\psi,c).
\end{align*}
 \end{proof}

Now we are in a position to give a proof of Theorem \ref{invariant formula for minimal conductor representation} by using Lemma 
\ref{Lemma 5.2.17}.

\begin{proof}[{\bf Proof of Theorem \ref{invariant formula for minimal conductor representation}}]
By the given conditions: 
 $\rho=\rho(X,\chi_K)$ is a minimal conductor 
 Heisenberg representation of the absolute Galois group $G_F$ of a local field $F/\bbQ_p$ of dimension
 $m$ which is prime to $p$. This means we are in the situation: $\rho=\rho(X,\chi_K)=\rho(X_\eta,\chi_K)$,
 where $\eta$ is a character of $U_F/U_F^1$, and $\rm{dim}(\rho)=\#\eta=m$.
 
 Since $\rho$ is of minimal conductor, we have $a(\rho_0)=m$. Then from Remark \ref{Remark 5.1.22} we have $a(\chi_K)=1$.
 
 Now we choose $E/F\subset K/F$ a totally ramified cyclic subextension of degree $[E:F]=m$, hence $k_E=k_F$
 the same residue fields, and $K/E$ is unramified of degree $m$. Then we can write $\rho=\rm{Ind}_{E/F}(\chi_E)$, and 
 $a(\chi_E)=1$. Again, from Proposition \ref{Proposition arithmetic form of determinant} we have 
 $$\det(\rho)(x)=\Delta_{E/F}(x)\cdot \chi_K\circ N_{K/E}^{-1}(x)\quad\text{for all $x\in F^\times$}.$$
Then for all $x\in F^\times$, we can write 
$$\chi_K\circ N_{K/E}^{-1}(x)=\chi_E(x)=\Delta_{E/F}(x)\cdot \det(\rho)(x).$$
This is true for all subextensions\footnote{In $K/F$ of type $\bbZ_m\times\bbZ_m$ any cyclic subextension $E/F$ in $K/F$
of degree $m$ will correspond to a maximal isotropic subgroup. But we restrict to choosing $E$
totally ramified or unramified.} $E/F$ in $K/F$ which are cyclic of degree $m$.

Now we come to in our particular choice: $\rho=\rm{Ind}_{E/F}(\chi_E)$, with $a(\chi_E)=1$ and $E/F$ is totally ramified.
We can write 
\begin{align*}
 W(\rho,\psi)
 &=W(\rm{Ind}_{E/F}(\chi_E),\psi)=\lambda_{E/F}(\psi)\cdot W(\chi_E,\psi\circ\rm{Tr}_{E/F})\\
 &=\lambda_{E/F}(\psi)\cdot q_E^{-\frac{1}{2}}\chi_E(c_E)\sum_{x\in U_E/U_E^1}\chi_E^{-1}(x)(c_E^{-1}\psi\circ\rm{Tr}_{E/F})(x),
\end{align*}
 where $v_E(c_E)=1+n(\psi\circ\rm{Tr}_{E/F})=e_{E/F}(1+n(\psi))$. This implies that we can choose $c_F\in F^\times$ such that 
 $\nu_F(c_F=c_E)=1+n(\psi)$. 
 Let $E_1/F$ be the unramified subextension in $K/F$, then for each $\epsilon\in U_F$, we have 
 $N_{K/E}^{-1}(\epsilon)=N_{E_1/F}^{-1}(\epsilon)$ where $N_{E_1/F}:=N_{K/E}|_{E_1}$. Since $E/F$ is totally ramified, we have 
 $q_E=q_F$. And when $x\in F^\times$, we have $\rm{Tr}_{E/F}(x)=mx$.
 
 Then the above formula rewrites:
 \begin{align*}
  W(\rho,\psi)
  &=\lambda_{E/F}(\psi)\cdot q_F^{-\frac{1}{2}}\chi_K\circ N_{K/E}^{-1}(c_F)\sum_{x\in k_F^\times}
  (\chi_K\circ N_{K/E}^{-1})^{-1}(x)(c_F^{-1}\psi)(mx)\\
  &=\lambda_{E/F}(\psi)\cdot q_F^{-\frac{1}{2}}\Delta_{E/F}(c_F)\det(\rho)(c_F)\sum_{x\in k_F^\times}
  (\chi_K\circ N_{E_1/F}^{-1})^{-1}(x)(c_F^{-1}\psi)(mx)\\
   &=\lambda_{E/F}(\psi)\Delta_{E/F}(c_F)\cdot\left(\det(\rho)(c_F)q_F^{-\frac{1}{2}}\sum_{x\in k_F^\times}
  (\chi_K\circ N_{E_1/F}^{-1})^{-1}(x)(c_F^{-1}\psi)(mx)\right)\\
  &=R(\psi,c)\cdot L(\psi,c),
 \end{align*}
where $c_F=c\in F^\times$ with $\nu_F(c)=1+n(\psi)$,
$R(\psi,c)=\lambda_{E/F}(\psi)\Delta_{E/F}(c)$, and 
$$L(\psi,c)=\det(\rho)(c_F)q_F^{-\frac{1}{2}}\sum_{x\in k_F^\times}
  (\chi_K\circ N_{E_1/F}^{-1})^{-1}(x)(c^{-1}\psi)(mx).$$
Now it is clear that $L(\psi,c)$ depends on $c$ but not on the totally ramified cyclic extension $E/F$ which we have chosen.

Again we know that $\lambda_{E/F}(\psi)$ is a fourth root of unity and $\Delta_{E/F}(c)\in\{\pm 1\}$. Therefore it is easy to see 
that $R(\psi,c)$ is a fourth root of unity. So to call our expression 
$$W(\rho,\psi)=R(\psi,c)\cdot L(\psi,c)$$
is invariant, we are left to show $R(\psi,c)$ does not depend on the the totally ramified cyclic subextension $E/F$ in $K/F$.

Moreover, we can write (cf. Lemma 3.2 of \cite{SAB1}) here 
$$R(\psi,c)=\lambda_{E/F}(\psi)\Delta_{E/F}(c)=\lambda_{E/F}(c\psi)=\lambda_{E/F}(\psi'),$$
where $\psi'=c\psi$, hence $n(\psi')=\nu_F(c)+n(\psi)=1+n(\psi)+n(\psi)=2n(\psi)+1$.

When $m(=[E:F])$ is odd, we have $\lambda_{E/F}(\psi')=1$, hence $R(\psi,c)=\lambda_{E/F}(c\psi)=1$. Thus in the odd case 
$R(\psi,c)$ is independent of the choice of the totally ramified subextension $E/F$ in $K/F$.

When $m$ is even, we have 
\begin{align*}
 R(\psi,c)
 &=\lambda_{E/F}(\psi')=\lambda_{E/E'}(\psi'')\cdot \lambda_{E'/F}^{[E:E'']}\\
 &=\lambda_{E'/F}(\psi')^{\pm 1},
\end{align*}
where $[E',F]$ is the $2$-primary part of $m$, hence $[E:E']$ is odd. Here the sign only depends on 
$m$ but not on $E$. So we can restrict to the case where $m=[E:F]$ is a power of $2$. Let $E_2/F$ be the unique quadratic subextension
in $E/F$. Since $E/F$ is a cyclic tame extension, from Theorem \ref{Theorem 2.5}, we obtain:
\begin{equation}
 \lambda_{E/F}(\psi')=\begin{cases}
                       \lambda_{E_2/F}(\psi') & \text{if $[E:F]\ne 4$}\\
                       \beta(-1)\cdot\lambda_{E_2/F}(\psi') & \text{if $[E:F]=4$},
                      \end{cases}
\end{equation}
where $\beta$ is the character of $F^\times/\cN_{E/F}$ of order $4$.

Since here $n(\psi')=2 n(\psi)+1$ is {\bf odd}\footnote{If $n(\psi')$ is even, 
then from the table of the Remark 5.10 of \cite{SAB1}, $\lambda_{E_2/F}(\psi')=-\lambda_{E_2'/F}(\psi')$,
where $E_2'/F$ be the totally ramified quadratic extension different from $E_2/F$.
Therefore $\lambda_{E/F}(\psi')$ depends on $\psi'$.},
from Remark 5.10 of \cite{SAB1}
we can tell that $\lambda_{E_2/F}(\psi')$ is invariant.\\
Finally, we have to see that $\beta(-1)$ does not depend on $E$ if $[E:F]=4$. 

Since $E/F$ is totally ramified of degree $4$, 
we have $F^\times=U_F\cdot N$, hence $F^\times/N=U_F N/N=U_F/U_F\cap N\cong\bbZ_4$, where $N=N_{E/F}(E^\times)$.
Again $U_F^1\subset U_F$, and $U_F^1\subset N$, hence $U_{F}^{1}\subset N\cap U_F\subset U_F$. We know that 
$U_F/U_F^1$ is a cyclic group. Therefore $N\cap U_F$ is determined by its index in $U_F$, which does not 
depend on $E$. Hence,
$U_F\cap N$ does not depend on $E$.

We also know that there are two characters of $U_F/U_F\cap N$ of order $4$, and they are inverse to each other. Then 
$$\beta(-1)=\beta(-1)^{-1}=\beta^{-1}(-1)$$ 
is the same in both cases. 
Since 
$\beta$ is the character which corresponds to $E/F$ by class field theory, we can say $\beta$ is a character of 
$F^\times/U_F^1$, hence $a(\beta)=1$. It clearly shows that $\beta(-1)$ does not depend on $E$. So we can conclude that 
$R(\psi,c)$ does not depend on $E$.



Thus our expression $W(\rho,\psi)=R(\psi,c)\cdot L(\psi,c)$ does not depend on the choice of the totally ramified
cyclic subextension $E/F$ in $K/F$. Moreover we notice that we have the transformation rules 
$$R(\psi,\epsilon c)=\Delta(\epsilon)R(\psi,c),\qquad L(\psi,\epsilon c)=\Delta(\epsilon)L(\psi,c),$$
for all $\epsilon\in U_F$. Again $\Delta(\epsilon)^2=1$, hence the product $R(\psi,\epsilon c)\cdot L(\psi,\epsilon c)=
R(\psi,c)\cdot L(\psi,c)=W(\rho,\psi)$ does not depend on the choice of $c$.

Therefore, finally, we can conclude our formula $W(\rho,\psi)=R(\psi,c)L(\psi,c)$ is an invariant expression.

\end{proof}

Now let $\rho=\rho(X,\chi_K)$ be a Heisenberg representation of dimension prime to $p$ but the conductor of $\rho$
is {\bf not} minimal. In the following theorem we give an invariant formula of $W(\rho,\psi)$.

\begin{thm}\label{Theorem invariant for non minimal representation}
 Let $\rho=\rho(X_\rho,\chi_K)$ be a Heisenberg representation of the absolute Galois group $G_F$ of a local field $F/\bbQ_p$
 of dimension $m$ prime to $p$. Let $\psi$ be a nontrivial additive character of $F$. 
 Suppose that the conductor of $\rho$ is not minimal, $\rho=\rho_0\otimes\widetilde{\chi_F}$ and $a(\rho)=m\cdot a(\chi_F)$, 
 where $\widetilde{\chi_F}:W_F\to\bbC^\times$ corresponds to $\chi_F:F^\times\to\bbC^\times$, and 
 $h=a(\chi_F)\ge 2$.\\
 {\bf Case-1:} If $m$ is odd, then
 \begin{enumerate}
  \item when $1+m(h-1)=2d$ is even, we have
  $$W(\rho,\psi)=\det(\rho)(c) \psi(mc^{-1}),$$
  \item when $1+m(h-1)=2d+1$ is odd, we have
  $$W(\rho,\psi)=\det(\rho)(c)\cdot H(\psi,c),$$
 \end{enumerate}
 where 
 $$H(\psi,c)=q_F^{-\frac{1}{2}}\sum_{y\in U_F^{h'}/U_F^{h'+1}}(\chi_K\circ N_{E_1/F}^{-1})^{-1}(y)(c^{-1}\psi)(my),$$
 and $h'=[\frac{h}{2}]$, where $[x]$ denotes the largest integer $\le x$.\\
 {\bf Case-2:} If $m$ is even, then 
 \begin{enumerate}
  \item when $h$ is odd, we have 
  $$W(\rho,\psi)=R(\psi,c)\cdot\det(\rho)(c)\cdot H(\psi,c),$$
  where $H(\psi,c)$ is the same as in Case-1(2).
  \item when $h$ is even, we have 
  $$W(\rho,\psi)=R(\psi,c)\cdot\det(\rho)(c)\cdot q_F^{\frac{1}{2}}\cdot\psi(c^{-1}m),$$
 \end{enumerate}
 where $R(\psi,c)=\lambda_{E/F}(\psi)\cdot \Delta_{E/F}(c)$.\\
Here $E_1/F$ is the maximal unramified subextension in $K/F$, and 
$E/F$ is a totally ramified cyclic subextension in $K/F$ and $c\in F^\times$ with $\nu_F(c)=h+n(\psi)$, and 
$$\chi_F(1+x)=\psi(x/c),\qquad \text{for all $x\in P_F^{h-h'}/P_F^h$}.$$
\end{thm}

\begin{proof}
{\bf Step-1:}
 By the given condition, $\rm{dim}(\rho)=m$ prime to $p$, and the Artin conductor 
 $a_F(\rho)=m h$ where $h\ge 2$, then from Lemma \ref{Lemma general conductor},
 we have $a(\chi_E)=mh-d_{E/F}=mh-m+1=1+m(h-1)$, where $E/F$ is a totally ramified cyclic subextension in $K/F$,
 and $\rho=\rm{Ind}_{E/F}(\chi_E)$.
 
Since by the given condition $\rho$ is not minimal conductor, we can write 
 \begin{equation}\label{eqn 55}
  \rho=\rho_0\otimes\widetilde{\chi_F},
 \end{equation}
 where $\rho=\rho_0(X,\chi_0)$ is a minimal conductor Heisenberg representation of dimension $m$, and 
 $\widetilde{\chi_F}: W_F\to\bbC^\times$ corresponds to $\chi_F:F^\times\to\bbC^\times$ by class field theory.\\
 Then we have $X_\rho=X_\eta$ for $\eta:U_F/U_F^1\to\bbC^\times$, $\#\eta=m$ and:
 $$\rho_0=\rm{Ind}_{E/F}(\chi_{E,0})\qquad\rho=\rm{Ind}_{E/F}(\chi_E),$$
 where $E/F$ is a cyclic totally ramified extension of degree $m$.\\
 Because of (\ref{eqn 55}) we may assume now that
 \begin{equation}\label{eqn 66}
  \chi_E=\chi_{E,0}\cdot (\chi_F\circ N_{E/F}), \quad a(\chi_{E,0})=1,\quad a(\chi_E)=a(\chi_F\circ N_{E/F})=1+m(h-1).
 \end{equation}
  From the first and second of the equalities (\ref{eqn 66}) we deduce 
 \begin{equation}\label{eqn 77}
  \chi_E|_{U_E^1}=(\chi_F\circ N_{E/F})|_{U_E^1}, \quad N_{E/F}(U_E^1)=U_F^1,
 \end{equation}
where the second equality holds because $E/F$ is totally ramified, and it implies that 
conversely $\chi_E|_{U_E^1}$ determines $\chi_F|_{U_F^1}$.\\
{\bf Step-2:} 
 Now for $d\ge 1$
we put:
$$A_E:=U_E^d/U_E^{d+1},$$
which we consider as a $\rm{Gal}(E/F)$-module. We also know that $A_E/I_{E/F}A_E\cong A_E^{\rm{Gal}(E/F)}$, where $I_{E/F}A_E$ is the 
augmentation with respect to the extension $E/F$.

We also know that for any finite extension $E/F$, we have 
 \begin{equation}\label{eqn intersection}
  U_E^d\cap F^\times=\begin{cases}
                      U_{F}^{\frac{d}{e_{E/F}}} & \text{if $e_{E/F}$ divides $d$}\\
                      U_{F}^{[\frac{d}{e_{E/F}}]+1} & \text{if $e_{E/F}$ does not divide $d$}.
                     \end{cases}
 \end{equation}
Again we also have 
$$A_E^{\rm{Gal}(E/F)}=U_E^n/U_E^{n+1}\cap F^\times=U_E^{d}\cap F^\times/U_E^{d+1}\cap F^\times.$$

{\bf Step-3:}
If $1+m(h-1)=2d+1$, then $\frac{d}{m}=\frac{h-1}{2}$. Let $h':=[\frac{h}{2}]$. If $A_E=U_E^d/U_E^{d+1}$, and $h$ is odd, then we have:
$$U_E^{d}\cap F^\times/U_E^{d+1}\cap F^\times=U_F^{\frac{h-1}{2}}/U_F^{\frac{h-1}{2}+1}=U_F^{h'}/U_F^{h'+1},$$
and if $h$ is even, hence $2$ does not divide $h-1$, then we can write 
$$U_E^{d}\cap F^\times/U_E^{d+1}\cap F^\times=U_F^{[\frac{h-1}{2}]+1}/U_F^{[\frac{h-1}{2}]+1}\cong\{1\}.$$
Since $A_E^{\rm{Gal}(E/F)}\cong A_E/I_{E/F}A_E$, we can uniquely write any element $x\in U_E^d/U_E^{d+1}$ as 
$x=yz$ where $y\in A_E^{\rm{Gal}(E/F)}$ and $z\in I_{E/F}A_E$. We also know that $U_E^d/U_E^{d+1}\cong k_E$, hence 
$|A_E|=|A_E^{\rm{Gal}(E/F)}|\cdot|I_{E/F}A_E|=q_E=q_F$. We also observe that when $h$ is even, we have 
$A_E^{\rm{Gal}(E/F)}\cong\{1\}$, hence $|A_E|=|I_{E/F}A_E|=q_F$. And when $h$ is odd, we have 
$A_E^{\rm{Gal}(E/F)}=U_F^{h'}/U_F^{h'+1}$, and hence $|A_E^{\rm{Gal}(E/F)}|=q_F$.
So this implies $|I_{E/F}A_E|=1$.

Now set:
$$S(\psi,c):=\sum_{x\in A_E}\chi_E^{-1}(x)(c^{-1}\psi)(\rm{Tr}_{E/F}(x)).$$
Then we can write
\begin{align*}
 S(\psi,c)
 &=\sum_{y\in A_E^{\rm{Gal}(E/F)},\;z\in I_{E/F}A_E}\chi_E^{-1}(yz)\cdot (c^{-1}\psi)(\rm{Tr}_{E/F}(yz))\\
 &=\sum_{y\in A_E^{\rm{Gal}(E/F)}}\sum_{z\in I_{E/F}A_E}\chi_E^{-1}(yz)(c^{-1}\psi)(\rm{Tr}_{E/F}(yz))\\
 &=|I_{E/F}A_E|\cdot \sum_{y\in A_E^{\rm{Gal}(E/F)}}\chi_E^{-1}(y)(c^{-1}\psi)(my)\\
 &=|I_{E/F}A_E|\cdot\sum_{y\in A_{E}^{\rm{Gal}(E/F)}}(\chi_K\circ N_{E_1/F}^{-1})^{-1}(y)(c^{-1}\psi)(my)\\
 &=\begin{cases}
    \sum_{y\in U_F^{h'}/U_{F}^{h'+1}}(\chi_K\circ N_{E_1/F}^{-1})^{-1}(y)(c^{-1}\psi)(my) & \text{when $h$ is odd}\\
    q_F\cdot (c^{-1}\psi)(m)=q_F\cdot\psi(mc^{-1}) & \text{when $h$ is even},
   \end{cases} 
\end{align*}
since $\chi_E(yz)=\chi_E(y)$ and $\rm{Tr}_{E/F}(yz)=y\rm{Tr}_{E/F}(z)=ym$.

{\bf Step-4:} Again, we have
 $\rho=\rm{Ind}_{E/F}(\chi_E)$. Then 
 $$W(\rho,\psi)=W(\rm{Ind}_{E/F}(\chi_E),\psi)=\lambda_{E/F}(\psi)\cdot W(\chi_E,\psi\circ\rm{Tr}_{E/F}).$$
 {\bf Case-1: Suppose that $m$ is odd:}\\
{\bf (1) When $a(\chi_E)=1+m(h-1)=2d$:} In this situation, $h$ must be even and we take $h=2h'$, hence $d=mh'-\frac{m-1}{2}$.
Since $m(h'-1)<d\le mh'$, we have  $P_E^d\cap F=P_F^{h'}$. 
Now we choose $c\in F^\times$ such that 
\begin{equation}\label{eqn 991}
 \chi_F(1+y)=\psi(c^{-1}y), \quad\text{for all $y\in P_F^{h-h'}/P_F^h$},
\end{equation}
hence $\nu_F(c)=a(\chi_F)+n(\psi)=h+n(\psi)$.
Now if we take an element 
$y_E\in P_E^{a(\chi_E)-d}=P_E^{d}$, then 
$\rm{Tr}_{E/F}(y_E)\in P_F^{h'}=P_F^{h-h'}$ because $m(h'-1)<d\le mh'=m(h-h')$.  Since $E/F$ is cyclic, 
from Proposition 1.1 on p. 68 of \cite{FV}, we have:
$$N_{E/F}(1+y_E)=1+\rm{Tr}_{E/F}(y_E)+N_{E/F}(y_E)+\rm{Tr}_{E/F}(\delta),$$
where $\nu_E(\delta)\ge 2d=a(\chi_E)$.
Then for all $y_E\in P_E^{a(\chi_E)-d}/P_E^{a(\chi_E)}$, we can write 
\begin{align}\label{eqn 1001}
 \chi_E(1+y_E)\nonumber
 &=\chi_F\circ N_{E/F}(1+y_E)=\chi_F(1+\rm{Tr}_{E/F}(y_E))\\
 &=\psi(c^{-1}\rm{Tr}_{E/F}(y_E))=(c^{-1}\psi_E)(y_E),
\end{align}
because $N_{E/F}(y_E)+\rm{Tr}_{E/F}(\delta)\in P_F^{h}$. This verifies that our choice of $c$ is right for applying Lamprecht-Tate 
formula for $W(\chi_E,\psi_E)$.

Now we apply Lamprecht-Tate formula (cf. Theorem 6.1.1 and its Corollary of \cite{SAB2}) and we obtain:
$$W(\chi_E,\psi_E)=\chi_E(c)\cdot (c^{-1}\psi_E)(1)=\Delta_{E/F}(c)\det(\rho)(c)\psi(mc^{-1}).$$
Therefore
\begin{align*}
 W(\rho,\psi)
 &=\lambda_{E/F}(\psi)\cdot W(\chi_E,\psi_E)\\
 &=\lambda_{E/F}(\psi)\cdot \Delta_{E/F}(c)\det(\rho)(c)\psi(mc^{-1})\\
 &=R(\psi,c)\cdot\det(\rho)(c)\cdot\psi(mc^{-1})\\
 &=\det(\rho)(c)\cdot\psi(mc^{-1}),
\end{align*}
where $R(\psi,c)=\lambda_{E/F}(\psi)\Delta_{E/F}(c)=\lambda_{E/F}(c\psi)=1$ because $E/F$ is an odd degree Galois extension.\\
{\bf (2). When $a(\chi_E)=1+m(h-1)=2d+1$:}
Since $m$ is odd, here $h$ must be odd. Let $h':=[\frac{h}{2}]$. Then from Step-3 we have 
$A_E^{\rm{Gal}(E/F)}=U_F^{h'}/U_{F}^{h'+1}$. Now if we choose $c\in F^\times$ such that 
$$\chi_F(1+y)=\psi(c^{-1}y),\qquad\text{for all $y\in P_F^{h-h'}/P_F^{h}$}.$$
Then this $c$ also satisfies the following relation
$$\chi_E(1+y_E)=\psi_E(c^{-1}y_E),\qquad \text{for all $y_E\in P_E^{a(\chi_E)-d}/P_E^{a(\chi_E)}$},$$
because $d=\frac{m(h-1)}{2}$, and hence $m(h'-1)<d\le mh'$.
Then by Lamprecht-Tate formula we have 
\begin{align*}
 W(\chi_E,\psi_E)
 &=\chi_E(c)\psi_E(c^{-1})q_E^{-\frac{1}{2}}\sum_{x\in P_E^d/P_E^{d+1}}\chi_E^{-1}(1+x)\cdot(c^{-1}\psi_E)(x)\\
 &=\chi_E(c)\cdot q_F^{-\frac{1}{2}}\sum_{x\in U_E^d/U_E^{d+1}}\chi_E^{-1}(x)\cdot (c^{-1}\psi)(\rm{Tr}_{E/F}(x))\\
 &=\Delta_{E/F}(c)\det(\rho)(c)q_F^{-\frac{1}{2}}
 \sum_{y\in U_F^{h'}/U_F^{h'+1}}(\chi_K\circ N_{E_1/F}^{-1})^{-1}(y)(c^{-1}\psi)(my),
\end{align*}
because $h$ is odd, and we use Step-3. 
Thus we obtain
\begin{align*}
 W(\rho,\psi)
 &=W(\rm{Ind}_{E/F}(\chi_E),\psi)=\lambda_{E/F}(\psi)\cdot W(\chi_E,\psi_E)\\
 &=\lambda_{E/F}(\psi)\cdot \Delta_{E/F}(c)\det(\rho)(c)q_F^{-\frac{1}{2}}
 \sum_{y\in U_F^{h'}/U_F^{h'+1}}(\chi_K\circ N_{E_1/F}^{-1})^{-1}(y)(c^{-1}\psi)(my).\\
 &=R(\psi,c)\cdot\det(\rho)(c)q_F^{-\frac{1}{2}}
 \sum_{y\in U_F^{h'}/U_F^{h'+1}}(\chi_K\circ N_{E_1/F}^{-1})^{-1}(y)(c^{-1}\psi)(my).\\
 &=\det(\rho)(c)q_F^{-\frac{1}{2}}
 \sum_{y\in U_F^{h'}/U_F^{h'+1}}(\chi_K\circ N_{E_1/F}^{-1})^{-1}(y)(c^{-1}\psi)(my),
\end{align*}
because $m$ is odd, hence $R(\psi,c)=\lambda_{E/F}(c\psi)=1$.\\

{\bf Case-2: Suppose that $m$ is even.} If $m$ is even, then $1+m(h-1)=2d+1$ is always an odd number and $d=\frac{m(h-1)}{2}$.
But here $h$ could be any number $\ge 2$, i.e., $h$ is not fixed, and we put $h':=[\frac{h}{2}]$.
This implies $m(h'-1)<d\le mh'$ and $P_E^{d}\cap F=P_F^{h'}$. Now we take $c\in F^\times$ such that (\ref{eqn 991})
holds, and this again satisfies equation (\ref{eqn 1001}). Therefore we can use Lamprecht-Tate formula and
we have two cases:\\
\begin{enumerate}
 \item When $h$ is odd, we are in the same situation of Case-1(2), and we have 
 $$W(\rho,\psi)=R(\psi,c)\cdot\det(\rho)(c)\cdot H(\psi,c).$$
 \item When $h$ is even, from Step-3 we know that $A_E^{\rm{Gal}(E/F)}\cong\{1\}$ and  
 \begin{align*}
  \sum_{x\in A_E}\chi_E^{-1}(x)(c^{-1}\psi)(\rm{Tr}_{E/F}(x))
  &=q_F\cdot \psi(mc^{-1}).
 \end{align*}
Therefore in this situation we have 
\begin{align*}
 W(\rho,\psi)
 &=R(\psi,c)\cdot\det(\rho)(c)q_F^{-\frac{1}{2}}\cdot\sum_{x\in A_E}\chi_E^{-1}(x)(c^{-1}\psi)(\rm{Tr}_{E/F}(x))\\
 &=R(\psi,c)\cdot\det(\rho)(c)q_F^{-\frac{1}{2}}\cdot q_F\psi(mc^{-1})\\
 &=R(\psi,c)\cdot\det(\rho)(c)q_F^{\frac{1}{2}}\psi(mc^{-1}).
\end{align*}

\end{enumerate}
Furthermore, in the proof of Theorem \ref{invariant formula for minimal conductor representation}, we observe that 
$R(\psi,c)$ does not depend on $E$. Hence our above computations are invariant.

This completes the proof.
 
\end{proof}

By using following lemma, without using $\lambda$-function we also can
give invariant formula for  $W(\rho)$, when $\rm{dim}(\rho)$ is prime to $p$, for sufficiently large conductor character $\chi_F$.

\begin{lem}[{\bf Deligne-Henniart, \cite{BH}, p. 190, Proposition 29.4(4)}]\label{Lemma Deligne-Henniart}
 Let $F$ be a non-archimedean local field and $\psi$ be a nontrivial additive character of $F$. Let $\rho$ be a finite dimensional 
 representation of $G_F$. There is a sufficiently large integer $m_\rho$ such that if $\chi_F$ is a character of $F^\times$ of 
 conductor $a(\chi_F)\ge m_\rho$ , then 
 \begin{equation}\label{eqn DH}
  W(\rho\otimes\chi_F,\psi)=W(\chi_F,\psi)^{\rm{dim}(\rho)}\cdot \det(\rho)(c),
 \end{equation}
for any $c:=c(\chi_F,\psi)\in F^\times$
such that $\chi_F(1+x)=\psi(c^{-1}x)$, $x\in P_F^{[\frac{a(\chi_F)}{2}]+1}$.
\end{lem}

By using the above Lemma \ref{Lemma Deligne-Henniart}, we obtain the following theorem.

\begin{thm}\label{Theorem using Deligne-Henniart}
  Let $\rho=\rho_0\otimes\widetilde{\chi_F}$ be a Heisenberg representation of $G_F$ of dimension $d$ with $gcd(d,p)=1$, where 
  $\rho_0=\rho_0(X_\eta,\chi_0)$ is a minimal conductor Heisenberg representation. 
  If $a(\chi_F)\ge m_\rho \ge 2$, a sufficiently large number which depends on $\rho$, then we have 
 \begin{equation}\label{eqn 5.4.9}
  W(\rho,\psi)=W(\rho_0\otimes\widetilde{\chi_F})=W(\chi_F,\psi)^d\cdot\det(\rho_0)(c),
 \end{equation}
where $\psi$ is a nontrivial additive character of $F$, and $c:=c(\chi_F,\psi)\in F^\times$, satisfies 
\begin{center}
 $\chi_F(1+x)=\psi(c^{-1}x)$ for all $x\in P_{F}^{[\frac{a(\chi_F)}{2}]+1}$.
\end{center}
\end{thm}
\begin{proof}

From Corollary \ref{Corollary U-isotropic} we know that all Heisenberg representation $\rho$ of $G_F$ of dimension prime to 
$p$ are precisely given as $\rho=\rho(X_\eta,\chi)$ for characters $\eta$ of $U_F/U_F^1$. 
Then from Remark \ref{Remark 5.1.22} we have here $a_K(\chi_0)=1$.
This implies that we always can choose a character $\chi_0$ of $K^\times$ with 
$a(\chi_0)=1$ such that all other $\chi_K$ are given as 
$$\chi_K=(\chi_F\circ N_{K/F})\cdot\chi_0,$$
for arbitrary characters $\chi_F$ of $F^\times$. Therefore the whole set of Heisenberg (U-isotopic) representations of $G_F$ of dimension
prime to $p$ is:
\begin{center}
 $\rho_0=\rho_0(G_K,\chi_0)$ and $\rho=\rho(G_K,\chi_K)$, 
 where $\chi_K=(\chi_F\circ N_{K/F})\cdot \chi_0$, and $\chi_F\in\widehat{F^\times}$ .
\end{center}
We also know that there are $d^2$ characters of $F^\times/{F^\times}^d$ such that 
$\rho_0\otimes\widetilde{\chi}=\rho_0$ (cf. \cite{Z2}, p. 303, Proposition 1.4). So we always have:
$$\rho=\rho_0\otimes\widetilde{\chi_F}=\rho_0\otimes\widetilde{\chi\chi_F},$$
where $\chi\in\widehat{F^\times/{F^\times}^d}$, and
$\widetilde{\chi_F}:W_F\to\bbC^\times$ corresponds to $\chi_F$ by class field theory.

 Let $\zeta$ be a $(q_F-1)$-st
root of unity. Since $U_F^1$ is a pro-p-group and $gcd(p,d)=1$, we have 
\begin{equation}\label{eqn 5.2.23}
  F^\times/{F^\times}^d=<\pi_F>\times<\zeta>\times U_F^1/<\pi_F^d>\times<\zeta>^d\times U_F^1\cong \bbZ_d\times\bbZ_d,
\end{equation}
that is, a direct product of two cyclic group of same order. Hence $F^\times/{F^\times}^d\cong\widehat{F^\times/{F^\times}^d}$.
Since ${F^\times}^d=<\pi_F^d>\times<\zeta>^d\times U_F^1$, and
$F^\times/{F^\times}^d\cong \bbZ_d\times\bbZ_d,$
we have $a(\chi)\le 1$ and $\#\chi$ is a divisor of $d$ for all 
$\chi\in\widehat{F^\times/{F^\times}^d}$. Now if we take a character $\chi_F$ of $F^\times$ conductor $\ge m_\rho\ge 2$, hence  
$a(\chi_F)\ge 2 a(\chi)$ for all $\chi\in \widehat{F^\times/{F^\times}^d}$. Then by 
using Deligne's formula (cf. \cite{D1}, Lemma 4.16) we have 
$$W(\chi_F\chi,\psi)^d=\chi(c)^d\cdot W(\chi_F,\psi)^d=W(\chi_F,\psi)^d,$$
where $c\in F^\times$ with $\nu_F(c)=a(\chi_F)+n(\psi)$, satisfies 
$$\chi_F(1+x)=\psi(c^{-1}x),\quad\text{for all $x\in F^\times$ with $2\nu_F(x)\ge a(\chi)$}.$$

Finally, by using Lemma \ref{Lemma Deligne-Henniart} we can write 
\begin{align*}
 W(\rho,\psi)
 &=W(\rho_0\otimes\widetilde{\chi_F\chi},\psi)= W(\chi_F\chi,\psi)^{\rm{dim}(\rho_0)}\cdot\det(\rho_0)(c(\chi_F,\psi))\\
 &=W(\chi_F,\psi)^d\cdot\det(\rho_0)(c).
\end{align*}

\end{proof}

\section{\textbf{Applications of Tate's root-of-unity criterion}}

Let $K/F$ be a finite
 Galois extension of the non-archimedean local field $F$, and $\rho:\mathrm{Gal}(K/F)\to \mathrm{Aut}_{\mathbb{C}}(V)$ a 
 representation of $\mathrm{Gal}(K/F)$ on a complex vector space $V$. 
 Let $P(K/F)$ denote the first {\bf wild} ramification group of $K/F$.
 Let $V^{P}$ be the subspace of all elements of $V$ fixed by $\rho(P(K/F))$. Then $\rho$ induces a representation:
 \begin{center}
  $\rho^{P}:\mathrm{Gal}(K/F)/P(K/F)\to\mathrm{Aut}_{\mathbb{C}}(V^{P})$.
 \end{center}
 Let $\overline{F}$ be an algebraic closure of the local field $F$, and $G_F=\rm{Gal}(\overline{F}/F)$ be the absolute Galois
 group for $\overline{F}/F$.
  Let $\rho$ be a representation of $G_F$.\\
{\bf Then by Tate, $W(\rho)/W(\rho^{P})$ is a root of a unity (cf. \cite{JT1}, p. 112, Corollary 4).}\\
Now let $\rho$ be an irreducible representation $G_F$, then either $\rho^P=\rho$, in which case
$\frac{W(\rho)}{W(\rho^P)}=1$, or else $\rho^P=0$, in this case from Tate's result we can say $W(\rho)$ is a root of unity.
Equivalently:\\
If $W(\rho)$ is not a root of unity then $\rho^P\ne 0$, hence $\rho^P=\rho$ because $\rho$ is irreducible. This means that 
all vectors $v\in V$ of the representation space are fixed under $P$ action on $V$. \\
In other words, if we consider $\rho$ as a 
homomorphism $\rho:G_F\to\rm{Aut}_\bbC(V)$ then the elements from $P$ are mapped to the identity, hence 
\begin{center}
 $\rho^P=\rho$ means $P\subset\rm{Ker}(\rho)$.
\end{center}
Therefore we can state the following lemma.
\begin{lem}
 If $\rho$ is an irreducible representation of $G_F$, such that the subgroup  $P\subset G_F$, of wild ramification
does {\bf not trivially} act on the representation space  $V$ (this gives $\rho^P\ne \rho$, i.e., $\rho^P=0$), 
then   $W(\rho)$  is a root of unity.
\end{lem}

Before going to our next results we need to recall some facts from class field theory.
Let $F$ be a non-archimedean local field. Let $F^{ab}$ be the maximal abelian extension of $F$ and $F_{nr}$ be the maximal 
unramified extension of $F$. Then by local class field theory there is a unique homomorphism
$$\theta_{F}:F^\times\to \rm{Gal}(F^{ab}/F)$$
having certain properties (cf. \cite{JM}, p. 20, Theorem 1.1). 
This local reciprocity map $\theta_F$ is continuous and injective with dense image. 
From class field theory we have the following commutative diagram  
$$\begin{array}{ccccccccc} &&& && v_F &&  \\
0 & \to & U_F & \to & F^\times & \to & \bbZ & \to & 0 \\
&& \quad\downarrow \theta_F && \quad\downarrow\theta_F && \quad\downarrow \rm{id} \\
0 & \to & I_F & \to & \rm{Gal}(F^{ab}/F) & \to & \widehat{\bbZ} & \to & 0, 
\end{array}$$
where $I_F:=\rm{Gal}(F^{ab}/F_{nr})$ is the inertia subgroup of $\rm{Gal}(F^{ab}/F)$, and $\rm{Gal}(F_{nr}/F)$ is identified  with 
$\widehat{\bbZ}$ (cf. \cite{CF}, p. 144). We also know that $\theta_F:U_F\to I_F$ is an isomorphism. Moreover the descending chain 
$$U_F\supset U_{F}^{1}\supset U_{F}^{2}\cdots$$
is mapped isomorphically by $\theta_F$ to the descending chain of ramification subgroups of $\rm{Gal}(F^{ab}/F)$ in the upper numbering.

Now let $I$ be the inertia subgroup of $G_F$. Let $P$ be the wild 
ramification subgroup of $G_F$.
Then we have $G_F\supset I\supset P$. Parallel with this we have 
$F^\times\supset U_F\supset U_{F}^{1}$. Then we have 
\begin{equation}\label{sequence 5.3.1}
 1\to I/P\cdot[G_F,G_F]\to G_F/P\cdot[G_F,G_F]\to G_F/I\to 1,
\end{equation}
and parallel 
\begin{equation}\label{sequence 5.3.2}
 1\to U_F/U_{F}^{1}\to F^\times/U_{F}^{1}\to F^\times/U_F\to 1.
\end{equation}
Now by class field theory the left terms of sequences (\ref{sequence 5.3.1}) and (\ref{sequence 5.3.2}) 
are isomorphic, but for the right terms we have $G_F/I$ is isomorphic to the total
completion of $\bbZ$
(because here $G_F/I$ is profinite group, hence compact). We also have  $F^\times/U_F=<\pi_F>\times U_F/U_F\cong\bbZ$.
Therefore sequence (\ref{sequence 5.3.2}) is dense in (\ref{sequence 5.3.1}) because $\bbZ$
is dense in the total completion $\widehat{\bbZ}$. But $\bbZ$ and $\widehat{\bbZ}$ have the same finite factor groups. 
{\bf As a consequence $F^\times/U_{F}^{1}$ is also dense in $G_F/P\cdot[G_F,G_F]$.}




Let $\rho$ be a Heisenberg representation of the absolute Galois group $G_F$.
In the following proposition we show that if $W(\rho)$ is not a root of unity, then $\rm{dim}(\rho)|(q_F-1)$, and 
$a_F(\rho)$ is not minimal.
\begin{prop}\label{Proposition 4.12}
 Let $F/\bbQ_p$ be a local field and let $q_F=p^s$ be the order of its finite residue field. If
$\rho=(Z_\rho,\chi_\rho)=\rho(X_\rho,\chi_K)$ is a Heisenberg representation of the absolute 
Galois group $G_F$ such that $W(\rho)$ is not a root of unity, 
then $dim(\rho)|(q_F-1)$ and $a_F(\rho)$ is not minimal.
\end{prop}

\begin{proof}
Let $P$ denote the wild ramification subgroup of $G_F$.
By Tate's root-of-unity criterion, we know that $\gamma:=\frac{W(\rho)}{W(\rho^P)}$ is a root of unity. If $W(\rho)$ is not a 
root of unity, then $\rho=\rho^P$, otherwise $W(\rho)$ must be a root of unity. Again $\rho^P=\rho$ implies 
$P\subset\rm{Ker}(\rho)\subset Z_\rho\subset G_F$. So $G_F/Z_\rho$ is a quotient of $G_F/P$, hence $F^\times/U_F^1$.
 
Moreover,  from the dimension formula (\ref{eqn dimension formula}), we have 
$$\rm{dim}(\rho)=\sqrt{[G_F:Z_\rho]}=\sqrt{[K:F]}=\sqrt{[F^\times:\cN_{K/F}]},$$
where $Z_\rho=G_K$ and $\rm{Rad}(X)=\cN_{K/F}$, hence $F^\times/N$ is a quotient group of $F^\times/U_F^1$.
 Therefore the alternating character $X_\rho$ induces an alternating 
character $X$ on $F^\times/U_F^1.$ We also know that  
$F^\times=<\pi_F>\times<\zeta>\times U_{F}^{1}$, where $\zeta$ is a root of unity of order $q_F-1$. This implies 
$F^\times/U_{F}^{1}=<\pi_F>\times<\zeta>$.
So each element $x\in
F^\times/U_F^1$ can be written as $x= \pi_{F}^a\cdot \zeta^b$, where $a,b\in\bbZ$. 
We now take $x_1=\pi_{F}^{a_1}\zeta^{b_1}, x_2=\pi_{F}^{a_2}\zeta^{b_2}\in F^\times/U_{F}^{1}$, where $a_i,b_i\in\bbZ(i=1,2)$, then 
\begin{align*}
 X(x_1,x_2)
 &= X(\pi_{F}^{a_1}\zeta^{b_1},\; \pi_{F}^{a_2}\zeta^{b_2})\\
 &= X(\pi_{F}^{a_1},\zeta^{b_2})\cdot X(\zeta^{b_1},\pi_{F}^{a_2})\\
 &=\chi_\rho([\pi_{F}^{a_1},\zeta^{b_2}])\cdot\chi_\rho([\zeta^{b_1},\pi_{F}^{a_2}]).
\end{align*}
But this implies  $X^{q_F-1}\equiv 1$ because $\zeta^{q_F-1}=1$,
which means that $X$ is actually an alternating character on  $F^\times/({F^\times}^{(q_F-1)} U_F^1),$ and therefore
$G_F/G_K$ is actually a quotient of $F^\times/({F^\times}^{(q_F-1)} U_F^1).$ 
We also know that $U_F^1$ is a pro-p-group and therefore 
$$U_F^1=(U_F^1)^{q_F-1}\subset F^\times.$$
Thus the cardinality of 
$F^\times/({F^\times}^{(q_F-1)} U_F^1)$ is $(q_F-1)^2$ because 
$$F^\times/({F^\times}^{(q_F-1)} U_F^1)\cong \bbZ/(q_F-1)\bbZ\times<\zeta>\cong \bbZ_{q_F-1}\times\bbZ_{q_F-1}.$$
Therefore $\rm{dim}(\rho)$ divides $q_F-1.$

Since $\rm{dim}(\rho)|q_F-1$, from Lemma \ref{Lemma dimension equivalent} the alternating character
$X_\rho$ is U-isotropic and $X_\rho=X_\eta$ for a character $\eta:U_F/U_F^1\to\bbC^\times$. Since $\rho=\rho(X_\eta,\chi_K)$
is U-isotropic, from Proposition \ref{Proposition conductor}, $a_F(\rho)$ is a multiple of $\rm{dim}(\rho)$. Moreover, by the given 
condition, $W(\rho)$ is not a root of unity, hence $a_F(\rho)$ is not minimal, otherwise if $a_F(\rho)$ is minimal, then from 
Lemma \ref{Lemma 5.2.12} $W(\rho)$ is a root of unity.

\end{proof}

\vspace{2cm}
\newpage

\textbf{Acknowledgements.} I would like to thank Prof E.-W. Zink for suggesting this problem and his constant 
valuable advice also I thank to my adviser Prof. Rajat Tandon for his continuous encouragement. I
extend my gratitude to Prof. Elmar Grosse-Kl\"{o}nne for providing very good mathematical environment during stay in
Berlin. I am also grateful to Berlin Mathematical School for their financial help. 


\end{document}